\documentclass[final,leqno]{siamltex}

\usepackage{amsfonts,amssymb} 
\usepackage{amsmath} 
\usepackage{color} 
\usepackage{graphicx} 
\usepackage{epsfig,mathrsfs} 
\usepackage[bf,SL,BF]{subfigure} 
\usepackage{fancyhdr} 
\usepackage{CJK} 
\usepackage{caption} 
\usepackage{wrapfig} 
\usepackage{cases} 
\usepackage{bm}
\usepackage{algorithm}
\usepackage{algorithmic}
\newtheorem{remark}{Remark}[section] 
\newtheorem{example}{Example}[section] 

\title{Uniform convergence of  multigrid finite element method  for  time-dependent Riesz  tempered fractional  problem
\thanks{This work was supported by NSFC 11601206 and NSFC 11601460.}}

\author{Minghua Chen\thanks{Corresponding author. School of Mathematics and Statistics, Gansu Key Laboratory of Applied Mathematics and Complex Systems,
 Lanzhou University, Lanzhou 730000, P.R. China  (Email: chenmh@lzu.edu.cn).}
        \and Weiping Bu \thanks{ School of Mathematics and Computational Science, Hunan Key Laboratory   for Computation and Simulation  in Science and Engineering,
        Xiangtan University, Xiangtan 411105, Hunan, P.R. China (Email: weipingbu@lsec.cc.ac.cn). }
       \and Wenya Qi\thanks{ School of Mathematics and Statistics, Gansu Key Laboratory of Applied Mathematics and Complex Systems,
 Lanzhou University, Lanzhou 730000, P.R. China (Email: qiwy16@lzu.edu.cn). }
 \and Yantao Wang
\thanks{ South University of Science and Technology of China,  Shenzhen 518055,  P.R. China   (wangyantao14@gmail.com).
}
 }

\begin{document}

\maketitle

\begin{abstract}
In this article a theoretical framework for the Galerkin finite element approximation to the time-dependent Riesz tempered fractional  problem is provided 
without the fractional  regularity assumption.
Because  the time-dependent problems should become easier to solve as the time step $\tau \rightarrow 0$, which
correspond to the mass matrix dominant  [R. E.  Bank  and  T.  Dupont, {\em  Math. Comp.}, 153 (1981), pp. 35--51].
Based on the introduced and analysis of the  fractional $\tau$-norm,   the uniform convergence estimates of the V-cycle multigrid method
with the time-dependent fractional problem   is strictly proved,
which means that  the convergence rate of the  V-cycle MGM is independent of the mesh size $h$ and the time step $\tau$. The numerical experiments are performed to verify the convergence
with only   $\mathcal{O}(N \mbox{log} N)$ complexity by the fast Fourier transform method, where $N$  is the number of the grid points.
To the best of our knowledge, this is the first proof for  the convergence rate of the V-cycle multigrid finite element method  with $\tau\rightarrow 0$.

\end{abstract}



\pagestyle{myheadings}
\thispagestyle{plain}
\markboth{~~}{MULTIGRID FEM FOR TIME-DEPENDENT  FRACTIONAL PROBLEM}

\section{Introduction}
In most practical problems, the real physical domain is bounded and/or    the observables involved in dynamical  have finite moments.
Then, based on the continuous time random walk (CTRW) model \cite{Metzler:00},  the Riesz tempered fractional diffusion  equation can be derived by tempering the probability
of large jump length of the L\'{e}vy flights   \cite{Baeumera:10,Cartea:07,Chen:15,del-Castillo-Negrete:09}, which describes the probability density function (PDF) of the truncated L\'{e}vy flights \cite{Koponen:95,Mantegna:94}.
In this article we investigate the multigrid  finite element method (FEM) for solving the  time-dependent Riesz tempered fractional  problem
\begin{equation}\label{1.1}
\begin{split}
P_t(x,t)-\mathbb{\nabla}_x^{\alpha,\lambda}P(x,t)&=\varphi(x,t)~~~\,\,{\rm in}~\,~\Omega\times (0,T],\\
     P(x,t)&=0~~~~~~~~\,~~\,{\rm on} ~\left({\mathbb R}\setminus \Omega\right)\times (0,T],\\
     P(x,0)&=\psi(x)~~~\,~~~{\rm on} ~~\Omega\times \{ t=0 \},
 \end{split}
\end{equation}
where we  assume $\Omega=(a,b)$ to be an open, bounded subset of ${\mathbb R}$ for some fixed time $T>0$.
The Riesz tempered   fractional derivative   is defined by  \cite{Cartea:07}
\begin{equation}\label{1.2}
\nabla_x^{\alpha,\lambda} P(x,t) =\kappa_{\alpha}\left[  {_{a}}\mathbb{D}_x^{\alpha,\lambda}+{ _{x}}\mathbb{D}_{b}^{\alpha,\lambda}\right]P(x,t)
\end{equation}
with $\lambda>0, \,\kappa_{\alpha}=-\frac{1}{2\cos(\alpha \pi/2)}>0$, $\alpha \in (1,2)$.
 The left and right tempered Riemann-Liouville fractional derivatives are, respectively, defined by
\begin{equation}\label{1.3}
\begin{split}
 {_{a}}\mathbb{D}_x^{\alpha,\lambda}u(x)
&={ _{a}}D_x^{\alpha,\lambda}u(x)-\lambda^\alpha u(x)-\alpha \lambda^{\alpha-1} \frac{\partial u(x) }{\partial x},\\
{ _{x}}\mathbb{D}_{b}^{\alpha,\lambda}u(x)
&={_{x}}D_{b}^{\alpha,\lambda}u(x)-\lambda^\alpha u(x) +\alpha \lambda^{\alpha-1} \frac{\partial u(x) }{\partial x}.
\end{split}
\end{equation}
Here the main term of the left and right tempered Riemann-Liouville fractional derivatives, respectively, are given in \cite{Chen:13,Chen:15}
\begin{equation}\label{1.4}
\begin{split}
{ _{a}}D_x^{\alpha,\lambda}u(x)
&=e^{-\lambda x}{ _{a}}D_x^{\alpha}[e^{\lambda x}u(x)]=\frac{e^{-\lambda x}}{\Gamma(2-\alpha)}\frac{d^2}{dx^2} \int_a^x  \frac{e^{\lambda \xi} u(\xi)}{(x-\xi)^{\alpha-1}}d\xi,\\
{_{x}}D_{b}^{\alpha,\lambda}u(x)
&=e^{\lambda x}{_{x}}D_{b}^{\alpha}[e^{-\lambda x}u(x)]=\frac{e^{\lambda x}}{\Gamma(2-\alpha)}\frac{d^2}{dx^2} \int_x^b  \frac{e^{-\lambda \xi} u(\xi)}{(\xi-x)^{\alpha-1}}d\xi,
\end{split}
\end{equation}
and  the left and right tempered Riemann-Liouville fractional integrals, respectively, are defined by \cite{Podlubny:99}
\begin{equation}\label{1.5}
\begin{split}
{ _{a}}D_x^{-\alpha,\lambda}u(x)
&=e^{-\lambda x}{ _{a}}D_x^{-\alpha}[e^{\lambda x}u(x)]=\frac{1}{\Gamma(\alpha)} \int_a^x  e^{-\lambda (x-\xi)} {(x-\xi)^{\alpha-1}}u(\xi)d\xi,\\
{_{x}}D_{b}^{-\alpha,\lambda}u(x)
&=e^{\lambda x}{_{x}}D_{b}^{-\alpha}[e^{-\lambda x}u(x)]=\frac{1}{\Gamma(\alpha)} \int_x^b  e^{-\lambda (\xi-x)} {(\xi-x)^{\alpha-1}}u(\xi)d\xi.
\end{split}
\end{equation}

To make a coercive bilinear form  for the model (see section \ref{section2.2} for details),
we can take $P=e^{\sigma t}u$ with $\sigma>0$ \cite{Celik:17,Friedman:64}.  Then (\ref{1.1}) can be rewrite as
\begin{equation}\label{1.6}
\left\{ \begin{array}
 {l@{\quad } l}
u_t(x,t)-\kappa_{\alpha}\left[  {_{a}}\mathbb{D}_x^{\alpha,\lambda}+{ _{x}}\mathbb{D}_{b}^{\alpha,\lambda}\right]u(x,t)
+\sigma u(x,t)\!\!\!\!\!\!\!\!\!&=f(x,t)~\,\!\,\,{\rm in} ~\Omega\times (0,T],\\
   ~~~~ ~~~~~~~~~~~~~~~~~~~~~ ~~~~ ~~~~~~~~~~~~~~~~~~~~~~~    u(x,t)\!\!\!\!&=0~~~~\,~~~~{\rm on} \left({\mathbb R}\setminus \Omega\right)\times (0,T],\\
  ~~~~ ~~~~~~~~~~~~~~~~~~~~~ ~~~~ ~~~~~~~~~~~~~~~~~~~~~\,\,~ u(x,0)\!\!\!\!&=\psi(x)~~~~{\rm on} ~\Omega\times \{ t=0 \}
 \end{array}
 \right.
\end{equation}
with $f(x,t)=e^{-\sigma t}\varphi(x,t)$.

There are already some important progress for numerically solving the FEM fractional problems \cite{Metzler:00}.
For example, the time-fractional and/or space-fractional problems are  discussed in \cite{Bu:14,Bu:17,Deng:08,Ervin:07,Li:13,Ren:17,Zeng:13}.
A theoretical framework with the fractional  regularity assumption for the Galerkin finite element approximation to the steady-state fractional model is first presented in \cite{Ervin:05}
and  intensive studied in  \cite{Acosta:17,Jin:16,Wang:13}.
Ros-Oton and Serra study the regularity up to the boundary of solutions to the Dirichlet problem for the fractional Laplacian with H\"{o}lder estimates \cite{Ros-Oton:17};
Jin et al. pointed out that there is still a lack of the regularity of   weak solution  in general  but  given  the   regularity results of {\em strong} solutions \cite{Jin:15};
and Ervin et al. investigated the regularity of the solution to the two-side fractional diffusion equation  \cite{Ervin:18}.
Recently,  the  tempered fractional problems  with the tempered fractional  regularity assumption  are discussed in
\cite{Celik:17,Deng:16}. Here  a theoretical framework for the Galerkin finite element approximation to the time-dependent Riesz tempered fractional  problem is presented
without  the fractional  regularity assumption.

When considering iterative solvers for the large-scale linear systems arising  from the approximation of elliptic partial differential equations (PDEs),
multigrid methods (MGM) (such as  V-cycle and W-cycle) are often  optimal order process \cite {Bank:81,Bramble:87,Hackbusch:85}.
The elegant theoretical framework and  uniform convergence of the  V-cycle MGM for second order elliptic equation  is well established in \cite{Bank:85,Bramble:87,Xu:02}.
The convergence rate independent of the number of levels is presented by multigrid finite difference method for elliptic equations with variable coefficients \cite{Jia:14}.
In the case of  multilevel matrix algebras, for  special prolongation operators  \cite{Fiorentino:96},
the  convergence rate   of the  V-cycle MGM is  derived in \cite{Arico:07,Arico:04,Bolten:15} for the elliptic PDEs.
Using the traditional  (simple)  prolongation operator,  for the time-dependent second elliptic problems,
the new convergence proofs for V-cycle MGM including multilevel linear systems are given in \cite{CD:17,CDS:17}.
For the time-independent fractional PDEs, based on the idea of \cite{Bank:85,Bramble:87,Brenner:08},  the  convergence rate  of the  V-cycle MGM is discussed in \cite{Jiang:15,Zhou:13} and
the nearly uniform convergence result is derived in \cite{CN:16}.
For the time-dependent fractional PDEs, the convergence rate of the two-grid method has been performed in \cite{Chen:14,CD:17,Pang:12} by following the ideas in
 \cite{Chan:98,Fiorentino:96}; and the convergence of the V-cycle MGM is investigated with  $\tau$ (time step) a positive constant \cite{Bu:15}.
Because  the time-dependent PDEs should become easier to solve as the time step $\tau \rightarrow 0$, which
correspond to the mass matrix dominant \cite{Bank:81}.
As far as we know,  the convergence rate of the V-cycle multigrid  finite element method has not been consider for a time-dependent PDEs with $\tau\rightarrow 0$.
In this paper, based on the introduced and  analysis of the  fractional $\tau$-norm,   the convergence rate of the V-cycle MGM  is strictly proved,
i.e.,  the uniform convergence of the  V-cycle MGM is independent of the mesh size $h$ and  $\tau$.
Moreover, the fast Toeplitz matrix-vector multiplication is utilized to lower the computational cost with only
$\mathcal{O}(N \mbox{log} N)$ complexity by the  fast Fourier transform (FFT) method, where $N$  is the number of the grid points.

The  outline of the paper is as follows. In  the next section,  we reivew the tempered fractional Sobolev space and prove a coercivity and a bondness for the   bilinear form.
A theoretical framework for the Galerkin finite element approximation to the time-dependent Riesz tempered fractional  problem is presented
without the fractional  regularity assumption in Section 3.
In Section 4, we first define the fractional $\tau$-norm and prove   the convergence estimates of the V-cycle MGM with time-dependent fractional PDEs.
Results of numerical experiments are reported and discussed in Section 5, in order to show the effectiveness of the presented schemes. Finally, we conclude the paper with some remarks.

\section{Preliminaries}
Throughout this article, $c$, $C$ and $C_i$, $i\geq 0$ will denote positive constants, not necessarily the same at different occurrences,
which are independent of the mesh size $h$ and the time step $\tau$.
Let $C_0^\infty(\Omega)$ denote the set of all function $u\in C^\infty(\Omega)$ that vanish outside a compact subset of $\Omega$.
For $\nu\geq 0$, let $H^{\nu}(\Omega)$ denote the Sobolev space of order $\nu$ on the interval $\Omega$ and $H_0^{\nu}(\Omega)$ the set
of functions in $H^{\nu}(\Omega)$ whose extension by 0 are in $H^{\nu}({\mathbb R})$.
This section is devoted to some results on tempered fractional Sobolev spaces and some properties on bilinear form.

\subsection{Tempered fractional Sobolev spaces}
Based on the idea of \cite{Celik:17,Deng:16,Ervin:05}, we further develop the abstract setting for the analysis of the approximation to the  tempered fractional diffusion equation.

\begin{definition}\label{definition2.1}
Let $\nu>0$, $\lambda>0$.  Define the semi-norm
$$|u|_{J_L^{\nu,\lambda}({\mathbb R})}:=||{_{-\infty}}D_x^{\nu,\lambda}u||_{L^2({\mathbb R})},$$
and norm
$$||u||_{J_L^{\nu,\lambda}({\mathbb R})}:=\left(||u||^2_{L^2({\mathbb R})}+|u|_{J_L^{\nu,\lambda}({\mathbb R})}^2\right)^{1/2},$$
where $J_L^{\nu,\lambda}({\mathbb R})$ denotes the closure of $C_0^\infty({\mathbb R})$ with respect to $||\cdot||_{J_L^{\nu,\lambda}({\mathbb R})}$.
\end{definition}

\begin{definition}\label{definition2.2}
Let $\nu>0$, $\lambda>0$.  Define the semi-norm
$$|u|_{J_R^{\nu,\lambda}({\mathbb R})}:=||{_{x}}D_{\infty}^{\nu,\lambda}u||_{L^2({\mathbb R})},$$
and norm
$$||u||_{J_R^{\nu,\lambda}({\mathbb R})}:=\left(||u||^2_{L^2({\mathbb R})}+|u|_{J_R^{\nu,\lambda}({\mathbb R})}^2\right)^{1/2},$$
where $J_R^{\nu,\lambda}({\mathbb R})$ denotes the closure of $C_0^\infty({\mathbb R})$ with respect to $||\cdot||_{J_R^{\nu,\lambda}({\mathbb R})}$.
\end{definition}

In the following analysis, we define a semi-norm for functions in $H^{\nu,\lambda}({\mathbb R})$ in terms of the Fourier transform.
We shall depart from the constant coefficient $\sqrt{2\pi}$ of the  inverse Fourier transform. This convention simplifies the appearance of results
such as the following Parseval's Formula \cite{Rudin:87}
\begin{equation}\label{2.1}
  \int_{\mathbb R} u(x)\overline{v(x)}dx= \int_{\mathbb R} \widehat{u}(\omega)\overline{\widehat{v}(\omega)}d\omega,
\end{equation}
where the bar denotes complex conjugation.

\begin{lemma} (\cite{Cartea:07,Chen:13,Chen:15})\label{lemma2.3}
Let $\nu>0$, $\lambda>0$. Let  $u$,~${ _{-\infty}}D_x^{\nu,\lambda}u$, ${_{x}}D_{\infty}^{\nu,\lambda}u$ and their Fourier transform belong to  $ L^1({\mathbb R})$. Then
\begin{equation*}
\mathcal{F}\left({ _{-\infty}}D_x^{\nu,\lambda}u(x)\right)
 =(\lambda-i\omega)^{\nu}\widehat{u}(\omega),
\end{equation*}
and
\begin{equation*}
\mathcal{F}\left({_{x}}D_{\infty}^{\nu,\lambda}u(x)\right)
  =(\lambda+i\omega)^{\nu}\widehat{u}(\omega),
\end{equation*}
where $\mathcal{F}$ denotes Fourier transform operator and $\widehat{u}(\omega)=\mathcal{F}(u)$, i.e.,
\begin{equation*}
    \widehat{u}(\omega)=\int_{{\mathbb R}}e^{i\omega x }u(x)dx.
 \end{equation*}
\end{lemma}

\begin{definition}\label{definition2.4}
Let $\nu>0$, $\lambda>0$.  Define the semi-norm
$$|u|_{H^{\nu,\lambda}({\mathbb R})}:=||\left(\lambda^2+|\omega|^2\right)^{\nu/2}\widehat{u}||_{L^2({\mathbb R})},$$
and norm
$$||u||_{H^{\nu,\lambda}({\mathbb R})}:=\left(||u||^2_{L^2({\mathbb R})}+|u|_{H^{\nu,\lambda}({\mathbb R})}^2\right)^{1/2},$$
where $H^{\nu,\lambda}({\mathbb R})$ denotes the closure of $C_0^\infty({\mathbb R})$ with respect to $||\cdot||_{H^{\nu,\lambda}({\mathbb R})}$.
\end{definition}

\begin{definition}\label{definition2.5}
Let $\nu>0$, $\lambda=0$.  Define the semi-norm
$$|u|_{H^{\nu}({\mathbb R})}:=||\,|\omega|^{\nu}\widehat{u}||_{L^2({\mathbb R})},$$
and norm
$$||u||_{H^{\nu}({\mathbb R})}:=\left(||u||^2_{L^2({\mathbb R})}+|u|_{H^{\nu}({\mathbb R})}^2\right)^{1/2}
=\left(1+|\omega|^{2\nu}\right)^{1/2}||\widehat{u}||_{L^2({\mathbb R})},$$
where $H^{\nu}({\mathbb R})$ denotes the closure of $C_0^\infty({\mathbb R})$ with respect to $||\cdot||_{H^{\nu}({\mathbb R})}$.
\end{definition}
\begin{lemma}\cite{Adams:75,Deng:16}\label{lemma2.6}
Let $\nu >0$ and $x\geq 0$. Then
\begin{equation*}
\begin{split}
&(1+x^\nu)\leq (1+x)^\nu\leq 2^{\nu-1}(1+x^\nu)~~~~\forall ~\nu \geq 1; \\
& 2^{\nu-1}(1+x^\nu) \leq (1+x)^\nu\leq (1+x^\nu)~~~~\forall ~0<\nu\leq 1.
\end{split}
\end{equation*}
\end{lemma}

\begin{lemma} \cite{Celik:17,Deng:16,Ervin:05}[Fractional Poincar\'{e}-Friedrichs]\label{lemma2.7}
For $u\in J_{L}^{\nu,\lambda}({\mathbb R})$ or $u\in J_{R}^{\nu,\lambda}({\mathbb R})$ with a compact subset of $\Omega$,  respectively, we have
$$||u||_{L^2({\mathbb R})}\leq C|u|_{J_L^{\nu,\lambda}({\mathbb R})}~~{\rm and }~~||u||_{L^2({\mathbb R})}\leq C|u|_{J_R^{\nu,\lambda}({\mathbb R})}.$$
\end{lemma}

\begin{theorem}\label{theorem2.8}
Let $\nu>0$, $\lambda>0$, $u\in J_{L}^{\nu,\lambda}({\mathbb R})$. Then
the space $J_L^{\nu,\lambda}({\mathbb R})$, $J_R^{\nu,\lambda}({\mathbb R})$, $H^{\nu,\lambda}({\mathbb R})$,  $H^{\nu}({\mathbb R})$
are equal with respect to the semi-norms and norms; and in fact
$$|u|_{H^{\nu,\lambda}({\mathbb R})}=|u|_{J_L^{\nu,\lambda}({\mathbb R})}= |u|_{J_R^{\nu,\lambda}({\mathbb R})},$$
and
$$C_1|u|_{J_L^{\nu,\lambda}({\mathbb R})}\leq |u|_{H^{\nu}({\mathbb R})}\leq C_2|u|_{J_L^{\nu,\lambda}({\mathbb R})}.$$
\end{theorem}
\begin{proof}
According to Definition \ref{definition2.1} and  Lemma \ref{lemma2.7}, there exists
$$ |u|_{J_L^{\nu,\lambda}({\mathbb R})}\leq ||u||_{J_L^{\nu,\lambda}({\mathbb R})}=\left(||u||^2_{L^2({\mathbb R})}+|u|_{J_L^{\nu,\lambda}({\mathbb R})}^2\right)^{1/2}
\leq \left(1+C^2 \right)^{1/2} |u|_{J_L^{\nu,\lambda}({\mathbb R})},$$
which means that we have norm equivalence of $||u||_{J_L^{\nu,\lambda}({\mathbb R})}$ and $|u|_{J_L^{\nu,\lambda}({\mathbb R})}$.
From  Definitions \ref{definition2.1}, \ref{definition2.2}, \ref{definition2.4}, \ref{definition2.5} and  Lemma \ref{lemma2.3}
and Parseval's Formula (\ref{2.1}), we obtain
$$|u|_{H^{\nu,\lambda}({\mathbb R})}= |u|_{J_R^{\nu,\lambda}({\mathbb R})}= |u|_{J_L^{\nu,\lambda}({\mathbb R})}=||\left(\lambda^2+|\omega|^2\right)^{\nu/2}\widehat{u}||_{L^2({\mathbb R})},$$
and
$||u||_{H^{\nu,\lambda}({\mathbb R})}=||u||_{J_R^{\nu,\lambda}({\mathbb R})}=||u||_{J_L^{\nu,\lambda}({\mathbb R})}.$

Next we shall  prove norm equivalence of $||u||_{H^{\nu}({\mathbb R})}$ and $|u|_{J_L^{\nu,\lambda}({\mathbb R})}$ (or  $||u||_{J_L^{\nu,\lambda}({\mathbb R})}$).
Using Lemma  \ref{lemma2.6} with  $\nu \geq 1$, we obtain
\begin{equation}\label{2.2}
(1+|\omega|^{2\nu})^\frac{1}{2}\leq (1+|\omega|^2)^\frac{\nu}{2}\leq 2^\frac{{\nu-1}}{2}(1+|\omega|^{2\nu})^\frac{1}{2}~~~~\forall ~\nu \geq 1,
\end{equation}
and it is easy  to get
\begin{equation}\label{2.3}
\begin{split}
&\frac{1}{\lambda^\nu}\left(  \lambda^2+|\omega|^2\right)^{\nu/2}
 \leq \left( 1+|\omega|^2\right)^{\nu/2}\leq \left(  \lambda^2+|\omega|^2\right)^{\nu/2}~~~~\forall ~\lambda\geq 1, \\
& \left(  \lambda^2+|\omega|^2\right)^{\nu/2}\leq \left( 1+|\omega|^2\right)^{\nu/2}\leq \frac{1}{\lambda^\nu}\left(  \lambda^2+|\omega|^2\right)^{\nu/2} ~~~~\forall ~0<\lambda\leq 1.
\end{split}
\end{equation}
Multiplying  (\ref{2.2}) and (\ref{2.3}) by $||\widehat{u}||_{L^2({\mathbb R})}$,  and using Definitions \ref{definition2.1}, \ref{definition2.2}, \ref{definition2.4}, \ref{definition2.5} and  Lemma \ref{lemma2.3}, there exists
\begin{equation*}
\begin{split}
\frac{1}{\lambda^\nu}|u|_{J_L^{\nu,\lambda}({\mathbb R})}
& \leq (1+|\omega|^2)^\frac{\nu}{2}||\widehat{u}||_{L^2({\mathbb R})}
 \leq 2^\frac{{\nu-1}}{2}(1+|\omega|^{2\nu})^\frac{1}{2}||\widehat{u}||_{L^2({\mathbb R})}\\
& = 2^\frac{{\nu-1}}{2} ||u||_{H^{\nu}({\mathbb R})}
\leq  2^\frac{{\nu-1}}{2} (1+|\omega|^2)^\frac{\nu}{2}||\widehat{u}||_{L^2({\mathbb R})}\\
&\leq 2^\frac{{\nu-1}}{2}\left(  \lambda^2+|\omega|^2\right)^{\nu/2}||\widehat{u}||_{L^2({\mathbb R})}
= 2^\frac{{\nu-1}}{2} |u|_{J_L^{\nu,\lambda}({\mathbb R})}~~~~\forall ~\lambda\geq 1.
\end{split}
\end{equation*}
Similarly, we have
\begin{equation*}
\begin{split}
&|u|_{J_L^{\nu,\lambda}({\mathbb R})}
 \leq (1+|\omega|^2)^\frac{\nu}{2}||\widehat{u}||_{L^2({\mathbb R})} \leq  2^\frac{{\nu-1}}{2} ||u||_{H^{\nu}({\mathbb R})}
 \leq  \frac{1}{\lambda^\nu}\cdot 2^\frac{{\nu-1}}{2} |u|_{J_L^{\nu,\lambda}({\mathbb R})} ~\forall ~0<\lambda\leq 1.
\end{split}
\end{equation*}
Therefore, we obtain the norm equivalence of $||u||_{H^{\nu}({\mathbb R})}$ and $|u|_{J_L^{\nu,\lambda}({\mathbb R})}$, moreover
\begin{equation}\label{2.4}
\begin{split}
|u|_{H^{\nu,\lambda}({\mathbb R})}=|u|_{J_L^{\nu,\lambda}({\mathbb R})}
\geq \min\left\{ 1, \lambda^\nu  \right \}||u||_{{H^{\nu}}({\mathbb R})}~~~\forall ~\lambda>0.
\end{split}
\end{equation}
From  \cite{Ervin:05}, we obtain
$||u||_{L^2({\mathbb R})}\leq C|u|_{H^{\nu}({\mathbb R})}$
and
$$ |u|_{H^{\nu}({\mathbb R})}\leq ||u||_{H^{\nu}({\mathbb R})}
\leq \left(1+C^2 \right)^{1/2} |u|_{H^{\nu}({\mathbb R})},$$
which means that we have norm equivalence of $||u||_{H^{\nu}({\mathbb R})}$ and $|u|_{H^{\nu}({\mathbb R})}$ or  $|u||_{J_L^{\nu,\lambda}({\mathbb R})}$.
The similar arguments can be performed as  stated above with $0<\nu\leq 1$, we omit it here.
The proof is completed.
\end{proof}

\subsection{Coercive and continuous of bilinear form}\label{section2.2}
Assume that $u$ is a sufficiently smooth function and $v\in C_0^\infty(\Omega)$, we obtain
\begin{equation*}
\begin{split}
\left(-\left( {_{a}}\mathbb{D}_x^{\alpha,\lambda}+{ _{x}}\mathbb{D}_{b}^{\alpha,\lambda} \right)u,v\right)
&=\left(-\left( {_{a}}D_x^{\alpha,\lambda}+{ _{x}}D_{b}^{\alpha,\lambda} \right)u,v\right)+2\lambda^\alpha(u,v)\\
&=-2\left( {_{a}}D_x^{\alpha/2,\lambda} u,{ _{x}}D_{b}^{\alpha/2,\lambda}v\right)+2\lambda^\alpha(u,v).
\end{split}
\end{equation*}
Thus, we can  define the associated bilinear form $b$: $H_0^{\alpha/2}(\Omega)\times H_0^{\alpha/2}(\Omega)\rightarrow \mathbb R$ as
\begin{equation}\label{2.5}
\begin{split}
b(u,v)
&=-2\kappa_{\alpha}\left( {_{a}}D_x^{\alpha/2,\lambda} u,{ _{x}}D_{b}^{\alpha/2,\lambda}v\right)+2\kappa_{\alpha}\lambda^\alpha(u,v)+\sigma (u,v).
\end{split}
\end{equation}
\begin{definition}\label{definition2.01}
Define the spaces
$J_{L,0}^{\nu,\lambda}(\Omega)$, $J_{R,0}^{\nu,\lambda}(\Omega)$, $H_0^{\nu,\lambda}(\Omega)$,  $H_0^{\nu}(\Omega)$
as the closures of  $C_0^{\infty}(\Omega)$ under their respective norms.
\end{definition}

 For a  simple and intuitive derivation, we denote
\begin{equation}\label{2.06}
  a \circ |u|^2_{H^{\nu,\lambda}({\mathbb R})}
:=\int_{\mathbb R} a \cdot (\lambda^2+\omega^2)^{\nu}|\widehat{u}|^2 d\omega.
\end{equation}

\begin{lemma}\label{lemma2.9}
Let $\nu>0$, $\lambda>0$, $u \in H_0^{\nu/2}(\Omega)$. Then
\begin{equation*}
\begin{split}
\left( { _{a}}D_x^{\nu/2,\lambda}u, {_{x}}D_{b}^{\nu/2,\lambda}u \right)
 =\cos(\nu \theta) \circ |u|^2_{H^{\nu/2,\lambda}({\mathbb R})}
\end{split}
\end{equation*}
with $\theta=\arctan\frac{|\omega|}{\lambda}\in \left[0,\frac{\pi}{2}\right).$
\end{lemma}
\begin{proof}
Let overbar denotes complex conjugate. It is easy to get
\begin{equation*}
\overline{(\lambda+i\omega)^{\nu/2}}=\left\{ \begin{array}{lll}
\displaystyle
\overline{(\lambda-i\omega)^{\nu/2}}\,e^{-i\nu\theta},~~~~~~\omega\geq0; \\
\overline{(\lambda-i\omega)^{\nu/2}}\,e^{i\nu\theta},~~~~~~~~\omega< 0
\end{array}
 \right.
\end{equation*}
with
$\theta=\arctan\frac{|\omega|}{\lambda}\in \left[0,\frac{\pi}{2}\right).$
It should be noted that  $\theta=\frac{\pi}{2}$ if $\lambda=0$, see in \cite{Ervin:05}.

Using Lemma \ref{lemma2.3} and  Parseval's Formula (\ref{2.1}), we have
\begin{equation*}
\begin{split}
&\left( { _{a}}D_x^{\nu/2,\lambda}u, {_{x}}D_{b}^{\nu/2,\lambda}u \right)\\
&= \left( (\lambda-i\omega)^{\nu/2}\widehat{u}, (\lambda+i\omega)^{\nu/2}\widehat{u}\right)\\
&=\int_{-\infty}^0(\lambda-i\omega)^{\nu/2}\widehat{u} \, \overline{(\lambda+i\omega)^{\nu/2}\widehat{u}}\,d\omega
  +\int_0^{\infty}(\lambda-i\omega)^{\nu/2}\widehat{u} \, \overline{(\lambda+i\omega)^{\nu/2}\widehat{u}}\,d\omega\\
&=\int_{\mathbb R} \cos(\nu \theta) \cdot (\lambda^2+\omega^2)^{\nu/2}|\widehat{u}|^2 d\omega
=\cos(\nu \theta) \circ |u|^2_{H^{\nu/2,\lambda}({\mathbb R})}.
\end{split}
\end{equation*}
The proof is completed.
\end{proof}
\begin{lemma}\label{lemma2.10}
Let $\nu>0$, $\lambda>0$, $u \in  H_0^{\nu/2}(\Omega)$. Then
\begin{equation*}
\begin{split}
||u||^2_{L^2(\Omega)}
 =\frac{\cos^\nu( \theta)}{\lambda^\nu}\circ |u|^2_{H^{\nu/2,\lambda}({\mathbb R})}
\end{split}
\end{equation*}
with $\theta=\arctan\frac{|\omega|}{\lambda}\in \left[0,\frac{\pi}{2}\right).$
\end{lemma}
\begin{proof}
From  Parseval's Formula (\ref{2.1}), it yields
\begin{equation*}
\begin{split}
||u||^2_{L^2(\Omega)}
&=\int_{\mathbb R}|\widehat{u}|^2 d\omega
= \frac{1}{\lambda^\nu}\int_{\mathbb R}\left(\frac{\lambda^2}{\lambda^2+\omega^2}\right)^{\nu/2} \cdot (\lambda^2+\omega^2)^{\nu/2}|\widehat{u}|^2 d\omega\\
&= \frac{1}{\lambda^\nu}\cos^\nu( \theta)\circ |u|^2_{H^{\nu/2,\lambda}({\mathbb R})},
\end{split}
\end{equation*}
where we use
\begin{equation*}
   \frac{\lambda^2}{\lambda^2+\omega^2}=\frac{1}{1+\tan^2\theta}=\cos^2( \theta).
\end{equation*}
The proof is completed.
\end{proof}
\begin{lemma}\label{lemma2.11}
Let $1<\alpha<2$, $\lambda>0$, $u \in   H_0^{\alpha/2}(\Omega)$ . Then
\begin{equation*}
\begin{split}
-2\left( {_{a}}D_x^{\alpha/2,\lambda} u,{ _{x}}D_{b}^{\alpha/2,\lambda}u\right)+2\lambda^\alpha(u,u)
 = 2 \left[\cos^\alpha( \theta)-\cos(\alpha \theta)\right]\circ |u|^2_{H^{\alpha/2,\lambda}({\mathbb R})}\geq 0
\end{split}
\end{equation*}
with $\theta=\arctan\frac{|\omega|}{\lambda}\in \left[0,\frac{\pi}{2}\right).$
\end{lemma}
\begin{proof}
According to Lemmas \ref{lemma2.9} and \ref{lemma2.10}, there exists
\begin{equation*}
\begin{split}
-2\left( {_{a}}D_x^{\alpha/2,\lambda} u,{ _{x}}D_{b}^{\alpha/2,\lambda}u\right)+2\lambda^\alpha(u,u)
=2 \left[\cos^\alpha( \theta)-\cos(\alpha \theta)\right]\circ |u|^2_{H^{\alpha/2,\lambda}({\mathbb R})}.
\end{split}
\end{equation*}
Let
\begin{equation}\label{2.6}
  f\left(  \theta\right):=\cos^\alpha( \theta)-\cos(\alpha \theta).
\end{equation}
Obviously, $f(\theta)=0$  if $\theta=0$. Next we prove $f(\theta)\geq 0$ if $\theta \in \left(0, \frac{\pi}{2}\right)$.
Since
\begin{equation*}
\begin{split}
  f'\left( \theta\right)
  =\alpha\left[ \sin(\alpha\theta)-\cos^{\alpha-1}(\theta)\sin(\theta) \right]
  \geq \alpha\sin(\theta)g(\theta)
\end{split}
\end{equation*}
with
$$g(\theta)=\cos((\alpha-1) \theta)-\cos^{\alpha-1}( \theta),$$
and
$$g'(\theta)=(\alpha-1)\sin(\theta)\left[\cos^{\alpha-2}(\theta)-   \frac{\sin((\alpha-1)\theta)}{\sin(\theta)}  \right]>0 ~{\rm with}~~1<\alpha<2. $$
Hence, we know that $f\left(  \theta\right)$ in (\ref{2.6}) is strictly increasing in $\left[0, \frac{\pi}{2}\right)$.
The proof is completed.
\end{proof}

\begin{lemma}\label{lemma2.12}
The bilinear form $b(\cdot,\cdot)$ is coercive over $H_0^{\alpha/2}(\Omega)$, i.e., there exists a constant $C_0$ such that
$$b(u,u)\geq C_0||u||^2_{{H_0^{\alpha/2}}(\Omega)},$$
where $\lambda>0$, $1<\alpha<2$ and
$$C_0=\min\left\{ 1, \lambda^{\alpha}  \right \}\cdot \min\left\{ 2\kappa_{\alpha}\left[\frac{1}{2^\alpha}-\cos\left(\frac{\alpha \pi}{3}\right)\right],  \frac{\sigma }{(2\lambda)^\alpha}\right \}>0.$$
\end{lemma}
\begin{proof}
According to (\ref{2.5}) and   Lemmas \ref{lemma2.11}, \ref{lemma2.10}, we have
\begin{equation*}
\begin{split}
b(u,u)=2\kappa_{\alpha} \left[\cos^\alpha( \theta)-\cos(\alpha \theta)+ \frac{\sigma }{2\kappa_{\alpha}\lambda^\alpha}\cos^\alpha( \theta)\right]\circ |u|^2_{H^{\alpha/2,\lambda}({\mathbb R})}.
\end{split}
\end{equation*}
From  (\ref{2.6}), we know that $f\left(  \theta\right)=\cos^\alpha( \theta)-\cos(\alpha \theta)\geq 0$  is strictly increasing in $\left[0, \frac{\pi}{2}\right)$,
which leads to
\begin{equation*}
\begin{split}
2\kappa_{\alpha}\left[\cos^\alpha( \theta)\!-\!\cos(\alpha \theta)+ \frac{\sigma }{2\kappa_{\alpha}\lambda^\alpha}\cos^\alpha( \theta)\right]
\geq  \frac{\sigma }{\lambda^\alpha}\cos^\alpha\left( \frac{\pi}{3}\right)= \frac{\sigma }{(2\lambda)^\alpha}>0~\forall \theta \in\!\left[0,\frac{\pi}{3}\right],
\end{split}
\end{equation*}
and
\begin{equation*}
\begin{split}
2\kappa_{\alpha}\left[\cos^\alpha( \theta)\!-\!\cos(\alpha \theta)+ \frac{\sigma }{2\kappa_{\alpha}\lambda^\alpha}\cos^\alpha( \theta)\right]
\geq  2\kappa_{\alpha}\left[\frac{1}{2^\alpha}-\cos\left(\frac{\alpha \pi}{3}\right)\right]>0~\forall \theta \in\!\left[\frac{\pi}{3},\frac{\pi}{2}\right).
\end{split}
\end{equation*}
According to (\ref{2.4}) and the above equations, there exists
\begin{equation*}
\begin{split}
b(u,u)
&\geq \min\left\{ 2\kappa_{\alpha}\left[\frac{1}{2^\alpha}-\cos\left(\frac{\alpha \pi}{3}\right)\right],  \frac{\sigma }{(2\lambda)^\alpha}\right \}\cdot |u|^2_{H^{\alpha/2,\lambda}({\mathbb R})}\\
&\geq \min\left\{ 2\kappa_{\alpha}\left[\frac{1}{2^\alpha}-\cos\left(\frac{\alpha \pi}{3}\right)\right],  \frac{\sigma }{(2\lambda)^\alpha}\right \}
\min\left\{ 1, \lambda^{\alpha}  \right \}||u||^2_{{H_0^{\alpha/2}}(\Omega)}.
\end{split}
\end{equation*}
The proof is completed.
\end{proof}

\begin{lemma}\label{lemma2.13}
The bilinear form $b(\cdot,\cdot)$ is   continuous on  $H_0^{\alpha/2}(\Omega)\times H_0^{\alpha/2}(\Omega)$ with  $1<\alpha<2$, i.e., there exists a constant $C_1$ such that
$$|b(u,v)|\leq C_1||u||_{{H_0^{\alpha/2}}(\Omega)}||v||_{{H_0^{\alpha/2}}(\Omega)}.$$
\end{lemma}
\begin{proof}
From (\ref{1.3}),  (\ref{2.5}) and Theorem \ref{theorem2.8}, we have
\begin{equation*}
\begin{split}
|b(u,v)|
&\leq 2\kappa_{\alpha}\left|\left( {_{a}}D_x^{\alpha/2,\lambda} u,{ _{x}}D_{b}^{\alpha/2,\lambda}v\right)\right|+\left(2\kappa_{\alpha}\lambda^\alpha+\sigma \right)\left|(u,v)\right|\\
&\leq 2\kappa_{\alpha} |u|_{J_{L,0}^{\alpha/2,\lambda}(\Omega)} |v|_{J_{L,0}^{\alpha/2,\lambda}(\Omega)}+\left(2\kappa_{\alpha}\lambda^\alpha+\sigma \right)||u||_{L^2(\Omega)}||v||_{L^2(\Omega)}\\
&\leq C_1||u||_{{H_0^{\alpha/2}}(\Omega)}||v||_{{H_0^{\alpha/2}}(\Omega)}.
\end{split}
\end{equation*}
The proof is completed.
\end{proof}

\section{Finite element  method for  tempered fractional  problem}
In this section a theoretical framework for the Galerkin finite element approximation to the time-dependent tempered fractional  problem is presented,
which does not require for the fractional  regularity assumption \cite{Celik:17,Deng:16,Ervin:05}.
The proposed method is based on a Crank-Nicolson scheme on time and Galerkin finite element in space for (\ref{1.6}).
This section is devoted to the stability analysis of the time-stepping scheme and the detailed error analysis of semidiscretization on time
and of full discretization.
\subsection{Stability analysis and error estimates for the semi-discrete scheme}
Let $T>0$, $\Omega=(a,b)$ and $t_n=n\tau$, $n=0,1,\ldots N$, where $\tau=\frac{T}{N}$ is the time steplength.
We set $u^n$ or $u^n(x)$  as an approximation of $u(x,t_n)$ and $f^{n-1/2}$  as an approximation of $f(x,t_{n-1/2})$ and denote $\underline{u}^n=\frac{1}{2}(u^n+u^{n-1})$.
We now turn to the Crank-Nicolson scheme, which will produce a second order accurate method in time, i.e.,
\begin{equation*}
  \frac{\partial u(x,t)}{\partial t}\Big|_{t=t_{n-1/2}}=\overline{\partial} u^n+r_\tau^n(x)
\end{equation*}
with
\begin{equation}\label{3.1}
|r_\tau^n(x)|\leq C_u \tau^2~~{\rm and }~~\overline{\partial} u^n=\frac{u^n(x)-u^{n-1}(x)}{\tau},
\end{equation}
where $C_u$ is a constant depending only on $u$.
Then we get the following variational formulation  of (\ref{1.6}): Find $u^n \in H_0^{\alpha/2}(\Omega)$ such that
\begin{equation}\label{3.2}
(\overline{\partial} u^n,\chi)+b(\underline{u}^n,\chi)=(f^{n-1/2},\chi) ~~\forall \chi \in H_0^{\alpha/2}(\Omega).
\end{equation}

\begin{theorem}\label{theorem3.1}
  The semi-discretized scheme  (\ref{3.2}) is unconditionally stable.
\end{theorem}
\begin{proof}
Let $\widetilde{u}^n$ be the approximate solution of $u^n$, which is the exact solution of the variational formulation (\ref{3.2}).
Taking $E^n=\widetilde{u}^n-u^n$. Then we get the following perturbation equation
\begin{equation*}
(\overline{\partial} E^n,\chi)+b(\underline{E}^n,\chi)=0
\end{equation*}
with $\underline{E}^n=\frac{1}{2}(E^n+E^{n-1})$. Taking $\chi=\underline{E}^n$, we obtain
 $$||E^n||_{L^2(\Omega)}\leq ||E^{n-1}||_{L^2(\Omega)}\leq ||E^0||_{L^2(\Omega)}.$$
The proof is completed.
\end{proof}

\begin{theorem}\label{theorem3.2}
Let $u$ be the exact solution of (\ref{1.6}) and $u^n$ be the solution of semi-discrete scheme (\ref{3.2}). Then we have the following
error estimates
$$||u(x,t_n)-u^n(x)||_{L^2(\Omega)}\leq 2 C_u(b-a)^{1/2}T\tau^2,$$
where $C_u$ is defined by  (\ref{3.1}) and  $(x,t_n)\in (a,b) \times (0,T]$ with $N\tau\leq T$.
\end{theorem}
\begin{proof}
Define $e^n=u(x,t_n)-u^n(x)$.
According to  $e^0=0$ and (\ref{3.1}), (\ref{3.2}), (\ref{1.6}), there exists
\begin{equation*}
(e^n-e^{n-1},\chi)+\tau b(\underline{e}^n,\chi)=(R_\tau^{n},\chi)
\end{equation*}
with $\underline{e}^n=\frac{1}{2}(e^n+e^{n-1})$ and  $||R_\tau^{n}||\leq c||\tau r_\tau^{n}||\leq C_u \tau^3$. Taking $\chi=\underline{e}^n$ in the  above equation, we get
\begin{equation*}
||e^{n}||_{L^2(\Omega)}^2-||e^{n-1}||_{L^2(\Omega)}^2\leq \left(R_\tau^{n},e^{n}+e^{n-1} \right).
\end{equation*}
Replacing $n$ with $s$ and summing up for $s$ from $1$ to $n$, there exists
\begin{equation}\label{3.3}
\begin{split}
||e^{n}||_{L^2(\Omega)}^2
\leq \!\sum_{s=1}^{n}\left(R_\tau^{s},e^{s}+e^{s-1} \right)
\leq \!\sum_{s=1}^{n}||R_\tau^{s}||_{L^2(\Omega)}\cdot\left(||e^{s}||_{L^2(\Omega)}+||e^{s-1}||_{L^2(\Omega)} \right).
\end{split}
\end{equation}
Suppose $m$ is chosen so that $||e^m||_{L^2(\Omega)}=\max\limits_{0\leq s\leq N}||e^s||_{L^2(\Omega)}$. Then
\begin{equation*}
\begin{split}
||e^{m}||_{L^2(\Omega)}^2
\leq\sum_{s=1}^{m}||R_\tau^{s}||_{L^2(\Omega)}  \cdot\left(||e^{m}||_{L^2(\Omega)}+||e^{m}||_{L^2(\Omega)} \right).
\end{split}
\end{equation*}
Hence $$||e^n||_{L^2(\Omega)}\leq  \max\limits_{0\leq s\leq N}||e^s||_{L^2(\Omega)} =||e^{m}||_{L^2(\Omega)}\leq 2\sum_{s=1}^{m}||R_\tau^{s}||_{L^2(\Omega)}\leq  2 C_u(b-a)^{1/2}T\tau^2.$$
The proof is completed.
\end{proof}

\subsection{Finite element approximation and error estimates for full discretization}
Let $S_h$ be a partition of $\Omega$ such that $\overline{\Omega}=\{ \cup K:K\in S_h \}$ with the uniform mesh size $h$.
Denote a finite dimensional subspace  $V_h\subset H_0^{\alpha/2}(\Omega)$ with a basis of the piecewise polynomials of degree
at most $r-1$  as
$$V_h=\{ v \in   H_0^{\alpha/2}(\Omega) \cap C^0(\overline{\Omega}):v|_K\in P_{r-1}(K)~ \forall K \in S_h \}.$$

Let us define the energy norm associated with (\ref{2.5}) as
\begin{equation*}
 ||u||_E:= \sqrt{b(u,u)}.
\end{equation*}
Form Lemmas \ref{lemma2.12} and \ref{lemma2.13}, we have  norm equivalence of $||\cdot||_{H_0^{\alpha/2}(\Omega)}$ and $||\cdot||_E$.

Before we start to discuss the time-dependent fractional  problem we shall briefly review some basic relevant material about the finite element method for the
corresponding stationary problem.
\begin{definition}[Variational Solution of Stationary Problem]
For $f\in H^{-\frac{\alpha}{2}}(\Omega)$, find $u \in H_0^{\frac{\alpha}{2}}(\Omega)$ such that
\begin{equation*}
  b(u,v)=(f,v)~~\forall v \in H_0^{\frac{\alpha}{2}}(\Omega).
\end{equation*}
\end{definition}
The approximate  problem is then to find  $u_h\in V_h$ such that
\begin{equation*}
  b(u_h,v_h)=(f,v_h)~~\forall   v_h \in V_h.
\end{equation*}
\begin{lemma}\cite{Celik:17}\label{lemma3.4}
Let $u$ be the exact solution to $b(u,v)=(f,v)$ and let $P_h$ be the orthogonal  projection from $H_0^{\alpha/2}(\Omega)$ to its finite dimensional subspace $V_h$. Then
\begin{equation*}
  b(P_hu-u,v)=0~~\forall v\in V_h.
\end{equation*}
\end{lemma}
Next we prove the following fractional  regularity  estimate of {\em strong} solutions,  which is the intermediate situation of weak solution and classical solutions \cite{Gilbarg:97}.
\begin{lemma}(Fractional regularity estimate)\label{lemma3.5}
Let $\sigma\geq 3\lambda^\alpha$ with $1<\alpha<2$ and $z\in  H^{\alpha}(\Omega)\cap  H_0^{\alpha/2}(\Omega)$,  $g(x)\in L^2(\Omega)$ solve
the following adjoint problem
\begin{equation}\label{3.4}
\begin{split}
-\left( {_{a}}\mathbb{D}_x^{\alpha,\lambda}+{ _{x}}\mathbb{D}_{b}^{\alpha,\lambda} \right)z(x)+\sigma z(x) &=g(x),~~~~x\in \Omega,\\
z(x)&=0,~~~~~~~~x\in {\mathbb R}  \setminus \Omega.
\end{split}
\end{equation}
Then there exists a positive constant $C_a$ such that
$$||z||_{ H^{\alpha}(\Omega)} \leq C_a||g||_{L^2(\Omega)}.$$
\end{lemma}
\begin{proof} From Lemma \ref{lemma2.11}, we have
\begin{equation*}
\begin{split}
\left(-\left( {_{a}}\mathbb{D}_x^{\alpha,\lambda}+{ _{x}}\mathbb{D}_{b}^{\alpha,\lambda} \right)z,z\right)
&=\left(-\left( {_{a}}D_x^{\alpha,\lambda}+{ _{x}}D_{b}^{\alpha,\lambda} \right)z,z\right)+2\lambda^\alpha(z,z)\\
&=-2\left( {_{a}}D_x^{\alpha/2,\lambda} z,{ _{x}}D_{b}^{\alpha/2,\lambda}z\right)+2\lambda^\alpha(z,z)\\
&=2 \left[\cos^\alpha( \theta)-\cos(\alpha \theta)\right]\circ |z|^2_{H^{\alpha/2,\lambda}({\mathbb R})}\geq 0.
\end{split}
\end{equation*}
According to the above equation and   Lemmas  \ref{lemma2.9}-\ref{lemma2.11} and (\ref{2.4}) with $\sigma\geq 3\lambda^\alpha$, we obtain
\begin{equation*}
\begin{split}
||g||^2_{L^2(\Omega)}
&=\big|\big|-\left( {_{a}}\mathbb{D}_x^{\alpha,\lambda}+{ _{x}}\mathbb{D}_{b}^{\alpha,\lambda} \right)z+\sigma z \big|\big|^2_{L^2(\Omega)}\\
&\geq \big|\big|\left( {_{a}}\mathbb{D}_x^{\alpha,\lambda}+{ _{x}}\mathbb{D}_{b}^{\alpha,\lambda} \right)z \big|\big|^2_{L^2(\Omega)}
     +\sigma^2\big|\big| z \big|\big|^2_{L^2(\Omega)}\\
&= \big|\big| {_{a}}D_x^{\alpha,\lambda}z \big|\big|^2_{L^2(\Omega)}+\big|\big|{ _{x}}D_{b}^{\alpha,\lambda} z \big|\big|^2_{L^2(\Omega)}
+2\left( { _{a}}D_x^{\alpha,\lambda}z, {_{x}}D_{b}^{\alpha,\lambda}z \right)\\
&\quad-4\lambda^\alpha\left( \left( {_{a}}\mathbb{D}_x^{\alpha,\lambda}+{ _{x}}\mathbb{D}_{b}^{\alpha,\lambda} \right)z,z\right)
     +\left(\sigma^2-4\lambda^{2\alpha}\right)\big|\big| z \big|\big|^2_{L^2(\Omega)}\\
&\geq \left[(2+ 2\cos(2\alpha \theta))
     +\left(\sigma^2-4\lambda^{2\alpha}\right)\frac{\cos^{2\alpha}( \theta)}{\lambda^{2\alpha}}\right] \circ |z|^2_{H^{\alpha,\lambda}({\mathbb R})} \\
&\geq  4f(\theta)|z|^2_{H^{\alpha,\lambda}(\Omega)}
\geq C||z||^2_{{H^{\alpha}}(\Omega)},
\end{split}
\end{equation*}
where  $f(\theta)=\cos^2(\alpha\theta)+\cos^{2\alpha}( \theta)>0.$ In fact,
if $f(\theta)=0$, it lead to $\theta=\frac{\pi}{2}$ and $\alpha=1$, which contradicts  $1<\alpha<2$.
The proof is completed.
\end{proof}
\begin{remark}\label{remark3.1}\end{remark}
\!\!\!\!\!\!In recent years, the optimal error estimate was provided under the assumption that the  weak solution has full regularity \cite{Celik:17,Deng:16,Ervin:05};  Ros-Oton and Serra study the regularity up to the boundary of solutions to the Dirichlet problem for the fractional Laplacian with H\"{o}lder estimates \cite{Ros-Oton:17};
Jin et al. pointed out that there is still a lack of the regularity of  {\em weak} solution  in general  but  given  the   regularity estimate of {\em strong} solutions \cite{Jin:15};
and Ervin et al. investigated the regularity of the {\em strong} solution to the two-side fractional diffusion equation  \cite{Ervin:18}.

\begin{lemma}(Approximation Property)\label{lemma3.6}
Let $u\in H_0^{\alpha/2}(\Omega)\cap H^r(\Omega)$ with  $r>\alpha/2$.
Then there exists the positive constant $C$ such that
$$ ||P_hu-u||_{L^2(\Omega)}+h^{\alpha/2} ||P_hu-u||_{{H_0^{\alpha/2}}(\Omega)} \leq Ch^{r}||u||_{H^r(\Omega)}.$$
\end{lemma}
\begin{proof}
For any $v \in V_h \subset H_0^{\alpha/2}(\Omega)$, from Lemmas \ref{lemma2.12}, \ref{lemma3.4}, \ref{lemma2.13},  there exists
\begin{equation*}
\begin{split}
 C_0||P_hu-u||^2_{{H_0^{\alpha/2}}(\Omega)}
 &\leq b(P_hu-u,P_hu-u) \\
 &= b(P_hu-u,v-u)+b(P_hu-u,P_hu-v)\\
 &\leq C_1||P_hu-u||_{{H_0^{\alpha/2}}(\Omega)}||v-u||_{{H_0^{\alpha/2}}(\Omega)},
\end{split}
\end{equation*}
which leads to
\begin{equation}\label{3.5}
\begin{split}
||P_hu-u||_{{H_0^{\alpha/2}}(\Omega)}
&\leq \frac{C_1}{ C_0}\inf_{v\in V_h}||v-u||_{{H_0^{\alpha/2}}(\Omega)}
\leq \frac{C_1}{ C_0}||I_hu-u||_{{H_0^{\alpha/2}}(\Omega)}\\
&\leq \frac{C_1C_2}{ C_0}h^{r-\alpha/2}||u||_{H^r(\Omega)}\leq C_3h^{r-\alpha/2}||u||_{H^r(\Omega)},
\end{split}
\end{equation}
where $I_h$ is the interpolation operator \cite{Brenner:08,Ervin:05}, i.e., for  $u \in H^r(\Omega)$ and $0\leq s \leq r$, there exists
$$||I_hu-u||_{{H^{s}}(\Omega)}\leq C h^{r-s}||u||_{H^r(\Omega)}. $$

For the error bound in $L_2$-norm we proceed by  a  Aubin-Nitsche technology.
Let $z$ be the solution to (\ref{3.4}). Then $z$  satisfies the variational formulation
\begin{equation*}
  b(v,z)=(P_hu-u,v) ~~\forall v \in  H_0^{\alpha/2}(\Omega)
\end{equation*}
with $g=P_hu-u$.
Using  Lemmas \ref{lemma3.4} and  \ref{lemma3.5}, we obtain
\begin{equation*}
\begin{split}
||P_hu-u||^2_{L^2(\Omega)}
&=b(P_hu-u,z)=b(P_hu-u,z-I_hz)\\
&\leq  C_1||z-I_hz||_{{H_0^{\alpha/2}}(\Omega)}||P_hu-u||_{{H_0^{\alpha/2}}(\Omega)}\\
&\leq  C_1C_2h^{\alpha/2}||z||_{H^\alpha(\Omega)}||P_hu-u||_{{H_0^{\alpha/2}}(\Omega)}\\
&\leq  C_1C_2C_ah^{\alpha/2}||P_hu-u||_{L^2(\Omega)}||P_hu-u||_{{H_0^{\alpha/2}}(\Omega)}.
\end{split}
\end{equation*}
Using the above equation and (\ref{3.5}), there exists
\begin{equation*}
\begin{split}
||P_hu-u||_{L^2(\Omega)}
\leq  C_1C_2C_ah^{\alpha/2}||P_hu-u||_{{H_0^{\alpha/2}}(\Omega)}
\leq  C_1C_2C_a  C_3h^{r}||u||_{H^r(\Omega)}.
\end{split}
\end{equation*}
The proof is completed.
\end{proof}

Now we give the finite element approximation of (\ref{3.2}): Find $u_h^n\in V_h$ such that
\begin{equation}\label{3.6}
(\overline{\partial} u_h^n,\chi)+b(\underline{u}_h^n,\chi)=(f_h^{n-1/2},\chi).
\end{equation}

Let
\begin{equation}\label{3.7}
\overline{\partial} u^n=\overline{\partial} P_hu^n+r_h^n(x)
=\frac{P_hu^n(x)-P_hu^{n-1}(x)}{\tau}+r_h^n(x).
\end{equation}
Combine (\ref{3.1}) and (\ref{3.7}), there exists
\begin{equation*}
  \frac{\partial u(x,t)}{\partial t}\Big|_{t=t_{n-1/2}}=\overline{\partial} P_hu^n+r_{\tau,h}^n(x)
\end{equation*}
with
\begin{equation}\label{3.8}
r_{\tau,h}^n(x)=r_{\tau}^n(x)+r_{h}^n(x).
\end{equation}

\begin{lemma}\label{lemma3.7}
The truncation error $r_{\tau,h}^n(x)$  is bounded by
$$||r_{\tau,h}^n||_{L^2(\Omega)}\leq C\left( \tau^2+ h^{r}\right).$$
\end{lemma}
\begin{proof}
Here, $r_{\tau}^n(x)$ is given in (\ref{3.1}), i.e.,
$$||r_\tau^n||_{L^2(\Omega)}\leq C_u \tau^2.$$
From Lemma \ref{lemma3.6}, there exists
\begin{equation}\label{3.9}
  ||P_hu-u||_{L^2(\Omega)} \leq C_1h^{r}||u||_{H^r(\Omega)}.
\end{equation}
Then using (\ref{3.7}) and (\ref{3.9}), we have
\begin{equation*}
\begin{split}
r_h^n(x)=(I-P_h)\overline{\partial} u^n=(I-P_h)\tau^{-1}\int_{t_{n-1}}^{t_n}u_tds
=\tau^{-1}\int_{t_{n-1}}^{t_n}(I-P_h)u_tds,
\end{split}
\end{equation*}
it yields
\begin{equation*}
\begin{split}
||r_h^n||_{L^2(\Omega)}\leq \tau^{-1}\int_{t_{n-1}}^{t_n}||(I-P_h)u_t||_{L^2(\Omega)}ds
\leq C_1h^{r}||u_t||_{H^r(\Omega)}.
\end{split}
\end{equation*}
Hence, from (\ref{3.8}), there exists
\begin{equation*}
||r_{\tau,h}^n||_{L^2(\Omega)}\leq ||r_{\tau}^n||_{L^2(\Omega)}+||r_{h}^n||_{L^2(\Omega)}
\leq  C_u \tau^2+C_1h^{r}||u_t||_{H^r(\Omega)}\leq C\left( \tau^2+ h^{r}\right).
\end{equation*}
The proof is completed.
\end{proof}

\begin{theorem}\label{theorem3.8}
Let $u$ be the exact solution of (\ref{1.6}) and $u_h^n$ be the solution of the full discretization scheme (\ref{3.6}). Then we have the following
error estimates
$$||u(x,t_n)-u_h^n||_{L^2(\Omega)}\leq  C\left( \tau^2+ h^{r}\right).$$
\end{theorem}

\begin{proof}
Let $\varepsilon^n=u_h^n-P_hu(x,t_{n})$ with $\varepsilon^0=0$.  Using  Lemmas \ref{lemma3.7}, \ref{lemma3.4} and (\ref{3.6}), we get the following error equation
\begin{equation*}
(\varepsilon^n-\varepsilon^{n-1},\chi)+\tau b(\underline{\varepsilon}^n,\chi)=(R_{\tau,h}^{n},\chi)
\end{equation*}
with $\underline{\varepsilon}^n=\frac{1}{2}(\varepsilon^n+\varepsilon^{n-1})$ and  $||R_{\tau,h}^{n}||_{L^2(\Omega)}\leq C_1||\tau r_{\tau,h}^{n}||_{L^2(\Omega)}\leq  C\tau\left( \tau^2+ h^{r}\right) $.
 Taking $\chi=\underline{\varepsilon}^n$ in the above equation, we get
\begin{equation*}
||\varepsilon^{n}||_{L^2(\Omega)}^2-||\varepsilon^{n-1}||_{L^2(\Omega)}^2\leq \left(R_{\tau,h}^{n},\varepsilon^{n}+\varepsilon^{n-1} \right).
\end{equation*}
Replacing $n$ with $s$ and summing up for $s$ from $1$ to $n$, there exists
\begin{equation*}
\begin{split}
||\varepsilon^{n}||_{L^2(\Omega)}^2
\leq \sum_{s=1}^{n}\left(R_{\tau,h}^{s},\varepsilon^{s}+\varepsilon^{s-1} \right).
\end{split}
\end{equation*}

It can be treated in the same way as  (\ref{3.3}), we have
\begin{equation*}
  ||u_h^n-P_hu(x,t_{n})||_{L^2(\Omega)}=||\varepsilon^n||_{L^2(\Omega)}\leq \max\limits_{0\leq s\leq N}||\varepsilon^s||_{L^2(\Omega)}\leq  C_2\left( \tau^2+ h^{r}\right).
\end{equation*}
According to the above equation and Lemma \ref{lemma3.6}, which  leads to
\begin{equation*}
\begin{split}
||u(x,t_n)-u_h^n||_{L^2(\Omega)}
&\leq ||u(x,t_n)-P_hu(x,t_n)||_{L^2(\Omega)}+||P_hu(x,t_n)-u_h^n||_{L^2(\Omega)}\\
&\leq C_1h^{r}||u||_{H^r(\Omega)} + C_2\left( \tau^2+ h^{r}\right)\leq C\left( \tau^2+ h^{r}\right).
\end{split}
\end{equation*}
The proof is completed.
\end{proof}

\section{Multigrid method for time-dependent tempered fractional  problem}
The time-dependent fractional  MGM can be treated as the elliptic equations arising at a  fixed time step \cite{Bu:15}.
However,  the time-dependent PDEs should become easier to solve as the time step $\tau \rightarrow 0$, which
correspond to the mass matrix dominant \cite{Bank:81}.
Therefore, we need to look for an estimate of other form in the fractional $\tau$-norm, which is   independent of the time step $\tau$.

\subsection{Multigrid method}
Let $V_k$ denote $C^0$ piecewise linear functions with the  uniform mesh size $h_k=\frac{1}{2}h_{k-1}$, i.e., $V_{k-1}\subset V_k, ~ k\geq1.$
\begin{definition}\label{definition4.1}
The mesh-dependent inner product $(\cdot,\cdot)_k$ on $V_k$ is defined by
$$(v,w)_k=h_k\sum_{i=1}^{n_k}v(x_i)w(x_i),$$
where $\{ x_i\}_{i=1}^{n_k}$ is the set of internal nodes of mesh grid with $n_k={\rm dim}V_k$.
\end{definition}
\begin{definition}\cite{Brenner:08}\label{definition4.2}
The coarse-to-fine operator $I_{k-1}^k: V_{k-1}\rightarrow V_k$
is taken to the natural injection, i.e.,
$$I_{k-1}^kv=v ~~\forall v \in V_{k-1}.$$
The fine-to-coarse operator  $I_k^{k-1}: V_k\rightarrow V_{k-1}$
is defined to be the transpose of $I_{k-1}^k$ with respect to the $(\cdot,\cdot)_{k-1}$ and $(\cdot,\cdot)_{k}$ inner products.
In other words,
$$(I_k^{k-1}w,v)_{k-1}=(w,I_{k-1}^kv)_{k}=(w,v)_k ~~\forall v\in V_{k-1}, w\in V_k.$$
\end{definition}
We rewrite (\ref{1.6}) as the following variational form
\begin{equation*}
(u_t,v)+b(u,v)=(f,v)~~\forall v\in {H_0^{\alpha/2}}(\Omega), ~~t\geq 0.
\end{equation*}
We use the  Crank-Nicolson scheme  for the above equation, there exists
\begin{equation}\label{4.1}
\tau^{-1}(u^n,v)+\frac{1}{2}b(u^n,v)=(g^{n-1},v)~~~\forall v\in {H_0^{\alpha/2}}(\Omega)
\end{equation}
with
$$(g^{n-1},v)=\tau^{-1}(u^{n-1},v)-\frac{1}{2}b(u^{n-1},v)+(f^{n-1/2},v).$$
For the simplicity of illustration, we rewrite (\ref{4.1})  as
\begin{equation}\label{4.2}
a_\tau(z,v)=:\tau^{-1}(z,v)+\frac{1}{2}b(z,v)=(g,v)~~~\forall v\in {H_0^{\alpha/2}}(\Omega)
\end{equation}
with $z=u^n$.

Then the discretized problems is : Find $w_k\in V_k$ such that
\begin{equation}\label{4.3}
a_{\tau}(w_k,v)=(g,v)~~\forall v\in V_k.
\end{equation}
The operator $A_{k,\tau}:V_k\rightarrow V_k$ is defined by
\begin{equation}\label{4.4}
(A_{k,\tau}w,v)_k=a_{\tau}(w,v)~~\forall v,w\in V_k.
\end{equation}
In terms of the operator $A_{k,\tau}$, the resulting systems (\ref{4.3}) can be written as
\begin{equation}\label{4.5}
A_{k,\tau}u_k=g_k.
\end{equation}
Since $A_{k,\tau}$ is symmetric positive definite with respect to $(\cdot,\cdot)_k$, we can define a scale of mesh-dependent norms $|||\cdot|||_{s,k,\tau}$ in the following way
\begin{equation}\label{4.6}
|||v|||_{s,k,\tau}:=\sqrt{(A_{k,\tau}^sv,v)_k}.
\end{equation}

\begin{definition}\label{definition4.3}
Let $P_k:{H_0^{\alpha/2}}(\Omega)\rightarrow V_k$ be the orthogonal projection with respect to $a_\tau(\cdot,\cdot)$. In other words, $P_kv\in V_k$ and
\end{definition}
\begin{equation}\label{4.7}
a_{\tau}(v,w)=a_{\tau}(P_kv,w)~~\forall w\in V_k.
\end{equation}

Let $K_k$ be the iteration matrix of the smoothing operator. In this work, we
 take $K_k$ to be  the weighted (damped) Jacobi iteration matrix
\begin{equation}\label{4.8}
  K_k=I-S_kA_{k,\tau},~~S_{k}:=S_{k,\eta}=\eta D_{k,\tau}^{-1}
\end{equation}
with a weighting  factor $\eta \in (0,1/2]$, and $D_{k,\tau}$ is the diagonal of $A_{k,\tau}$.
A multigrid process can be regarded as defining a sequence of operators $B_k:\mathcal{B}_k\mapsto \mathcal{B}_k$
which is an approximate inverse of $A_{k,\tau}$ in the sense that $||I-B_kA_{k,\tau}||$ is bounded away from one.
The V-cycle multigrid algorithm \cite{Bramble:87} is provided in Algorithm \ref{MGM}.
\begin{algorithm*}
\caption{ V-cycle Multigrid Algorithm: Define $B_1=A_{1,\tau}^{-1}$. Assume that $B_{k-1}:\mathcal{B}_{k-1}\mapsto \mathcal{B}_{k-1}$ is defined.
We shall now define $B_k:\mathcal{B}_{k}\mapsto \mathcal{B}_{k}$ as an approximate iterative solver for the equation  $A_{k,\tau}z=g$.}
\label{MGM}
\begin{algorithmic}[1]
\STATE Pre-smooth: Let $S_{k,\eta}$ be defined by (\ref{4.8}), $z_0=0$, $l=1: m_1$, and
  $$z_l=z_{l-1}+S_{k,\eta_{pre}}(g_k-A_{k,\tau}z_{l-1}).$$
\STATE Coarse grid correction: Denote $e^{k-1} \in \mathcal{B}_{k-1}$ as the approximate solution of the residual equation $A_{k-1}e=I_k^{k-1}(g-A_{k,\tau}z_{m_1})$
with the iterator $B_{k-1}$
$$e^{k-1}=B_{k-1}I_k^{k-1}(g-A_{k,\tau}z_{m_1}).$$
\STATE Post-smooth:~~$z_{m_1+1}=z_{m_1}+I_{k-1}^{k}e^{k-1}$, $l=m_1+2: m_1+ m_2+1$, and
$$z_l=z_{l-1}+S_{k,\eta_{post}}(g-A_{k,\tau}z_{l-1}).$$
\STATE Define ${\rm MG} (k,z_0,g):=B_kg=z_{m_1+m_2+1}$.
\end{algorithmic}
\end{algorithm*}

\subsection{Uniform convergence of the V-cycle MGM}
Based on the (\ref{4.2}), we   define the fractional $\tau$-norm
\begin{equation}\label{4.9}
||v||^2_{\tau,\alpha}=\tau^{-1}||v||_{L^2(\Omega)}^2+||v||^2_{{H^{\alpha}}(\Omega)}~~\forall v\in H^{\alpha}(\Omega).
\end{equation}
\begin{remark}
The fractional $\tau$-norm  reduces to $\tau$-norm \cite{Bank:81,Douglas:79} when $\alpha=1$.
\end{remark}

In order to estimate the spectral radius, $\rho(A_{k,\tau})$, of $A_{k,\tau}$, we first introduce the following lemmas.
\begin{lemma}\label{lemma4.4}
The bilinear form $a_\tau(u,v)$ is symmetrical, continuous and coercive. In other words, there exist the positive constants $C_2,C_3$ such that
$$a_\tau(u,u)\geq C_2||u||^2_{\tau,\alpha/2} ~~{\rm and}~~|a_\tau(u,v)|\leq C_3||u||_{\tau,\alpha/2}||v||_{\tau,\alpha/2}.$$
\end{lemma}
\begin{proof}
According to (\ref{4.2}) and Lemma \ref{lemma2.12}, there exists
\begin{equation*}
\begin{split}
a_\tau(u,u)=\tau^{-1}(u,u)+\frac{1}{2}b(u,u)\geq \tau^{-1}(u,u)+\frac{C_0}{2}||u||^2_{{H^{\alpha/2}}(\Omega)}\geq C_2||u||^2_{\tau,\alpha/2}
\end{split}
\end{equation*}
with $C_2=\min\{1,C_0/2\}$.
On the other hand, using  Lemma \ref{lemma2.13}, we have
\begin{equation*}
\begin{split}
|a_\tau(u,v)|
&\leq \tau^{-1}|(u,v)|+\frac{1}{2}|b(u,v)|\\
&\leq \left(1+\frac{1}{2}C_1\right)\left(\tau^{-1}||u||_{L^2(\Omega)} ||v||_{L^2(\Omega)} + ||u||_{{H^{\alpha/2}}(\Omega)}||v||_{{H^{\alpha/2}}(\Omega)}\right)\\
&\leq \left(1+\frac{1}{2}C_1\right)\Big\{\left(\tau^{-2}||u||^2_{L^2(\Omega)} ||v||^2_{L^2(\Omega)} + ||u||^2_{{H^{\alpha/2}}(\Omega)}||v||^2_{{H^{\alpha/2}}(\Omega)}\right)\\
&\qquad+\tau^{-1}||u||^2_{L^2(\Omega)}||v||^2_{{H^{\alpha/2}}(\Omega)}+\tau^{-1}||v||^2_{L^2(\Omega)}||u||^2_{{H^{\alpha/2}}(\Omega)}\Big\}^{1/2}\\
&= \left(1+\frac{1}{2}C_1\right)||u||_{\tau,\alpha/2}||v||_{\tau,\alpha/2}.
\end{split}
\end{equation*}
The proof is completed.
\end{proof}
\begin{lemma}\cite{Quarteroni:94}[Interpolation theorem]\label{lemma4.5}
Let $\Omega$ be an open set of $\mathbb{R}$ with a Lipshitz continuous boundary. Let $s_1<s_2$ be two real numbers, and $\mu=(1-\theta)s_1+\theta s_2$ with
$0\leq \theta\leq 1$. Then there exists a constant $C>0$ such that
$$||v||_\mu\leq C||v||_{s_1}^{1-\theta}||v||_{s_2}^\theta~~\forall v\in  H^{s_2}(\Omega).$$
\end{lemma}
\begin{lemma}[Spectral radius theorem]\label{lemma4.6}
Let $A_{k,\tau}$ be defined by (\ref{4.4}).  Then there exists a constant $C>0$ such that
$$\rho(A_{k,\tau}) \leq C(1+\tau^{-1}h^\alpha)h^{-\alpha}.$$
\end{lemma}
\begin{proof}
Form Lemmas \ref{lemma2.13},  \ref{lemma4.5} and inverse estimation of \cite{Quarteroni:94}, there exists
\begin{equation*}
\begin{split}
b(v,v)
&\leq  C_1||v||^2_{{H^{\alpha/2}}(\Omega)}\leq  C_1\left(C_2||v||^{1-\alpha/2}_{L^2(\Omega)}\cdot||v||^{\alpha/2}_{{H^{1}}(\Omega)} \right)^2\\
&\leq  C_1\left(C_2||v||^{1-\alpha/2}_{L^2(\Omega)}\cdot h^{-\alpha/2}||v||^{\alpha/2}_{L^2(\Omega)} \right)^2
\leq C_3h^{-\alpha}||v||_{L^2(\Omega)}^2.
\end{split}
\end{equation*}
Let $\Lambda$ be an eigenvalues of $A_{k,\tau}$ with eigenvector $v$.  From  the above equation and Lemmas \ref{lemma4.4}, \ref{lemma2.12}, we have
\begin{equation*}
\begin{split}
\Lambda(A_{k,\tau})
&=\frac{(A_{k,\tau}v,v)}{(v,v)}=\frac{a_{\tau}(v,v)}{(v,v)}\leq \frac{C_4||v||^2_{\tau,\alpha/2}}{||v||_{L^2(\Omega)}^2}\\
&\leq \frac{C_5\left(\tau^{-1}||v||_{L^2(\Omega)}^2+b(v,v)\right)}{||v||_{L^2(\Omega)}^2}
\leq C(1+\tau^{-1}h^\alpha)h^{-\alpha}.
\end{split}
\end{equation*}
The proof is completed.
\end{proof}

\begin{lemma}\label{lemma4.7}
Let $A_{k,\tau}=\{a_{i,j}\}_{i,j=1}^{n_k}$ be defined by (\ref{4.5}) with $\lambda=0$. Then
\begin{equation*}
\frac{\eta}{\rho(A_{k,\tau}) }(\nu_k,\nu_k) \leq (S_k\nu_k,\nu_k)\leq (A_{k,\tau}^{-1}\nu_k,\nu_k)
\quad \forall \nu_k\in V_k,
\end{equation*}
where  $S_k=\eta D_{k,\tau}^{-1}$, $\eta \in (0,1/2]$ and $D_{k,\tau}$ is the diagonal of $A_{k,\tau}$.
\end{lemma}
\begin{proof}
It is easy to  check that  $A_{k,\tau}$ is a weakly diagonally dominant symmetric  Toeplitz M-matrix \cite{Chen:15,Yue:17}, i.e., $A_{k,\tau}$ is a positive definite matrix with positive entries on the diagonal and nonpositive off-diagonal entries and  the diagonal element of a matrix is at least as large as the sum of the off-diagonal elements in the same row or column  \cite{Briggs:00}.
Then the similar arguments can be performed as Lemma 2.4 of \cite{CDS:17}, the desired result is obtained.
\end{proof}
\begin{remark}\end{remark}
For $\lambda=0$, the Riesz tempered   fractional derivative  (\ref{1.2})
reduces to the  fractional Laplace operator.
We conclude that the stiffness matrix $A_{k,\tau}$ of the {\em linear} finite element approximation on a uniform grid,
after proper scaling $h$, is the same as the stiffness matrix  of the finite difference scheme, see \cite{Chen:15,Yue:17}.
For $\lambda>0$, the analytical solution of the stiffness matrix are special function \cite{Miller:93}, from our numerical experiences,
we find that it is also a weakly diagonally dominant symmetric  Toeplitz M-matrix.

According to (\ref{4.6}), (\ref{4.9}) and Lemma  \ref{lemma4.4}, it is easy to get
\begin{equation}\label{4.10}
\begin{split}
&c||v||_{L^2(\Omega)}\leq  |||v|||_{0,k,\tau}\leq C||v||_{L^2(\Omega)},\\
& c||v||_{\tau,\alpha/2} \leq |||v|||_{1,k,\tau} \leq C ||v||_{\tau,\alpha/2}, \\
& c||A_{k,\tau}v||_{L^2(\Omega)}\leq |||v|||_{2,k,\tau}\leq C||A_{k,\tau}v||_{L^2(\Omega)}.
\end{split}
\end{equation}
\begin{lemma}(Generalized Cauchy-Schwarz inequality with $\tau$-norm)\label{lemma4.8}
For any   real number  $\theta$, it holds
\begin{equation*}
\begin{split}
|a_\tau (v,w)|\leq |||v|||_{1+\theta,k,\tau}|||w|||_{1-\theta,k,\tau}~~ \forall v,w \in V_k.
\end{split}
\end{equation*}
\end{lemma}
\begin{proof}
Let  $\lambda_i$ with  $1\leq i\leq n_k$ be the eigenvalues of the operator $A_{k,\tau}$  and $\psi_i$
 be the corresponding eigenfunction satisfying the orthogonal  relation $(\psi_i,\psi_j)_k=\delta_{i,j}.$
We can write  $v=\sum_{i=1}^{n_k}c_i\psi_i,w=\sum_{j=1}^{n_k}d_j\psi_j$.
From  (\ref{4.6}) and (\ref{4.10}), we  obtain
\begin{equation*}
\begin{split}
a_\tau (v,w)
&=(A_{k,\tau}v,w)=\left(\sum_{i=1}^{n_k}\lambda_ic_i\psi_i,\sum_{j=1}^{n_k}d_j\psi_j\right)\\
&=\sum_{i=1}^{n_k}\lambda_ic_id_i\leq \left(\sum_{i=1}^{n_k}c_i^2\lambda_i^{1+\theta}\right)^{1/2}\left(\sum_{i=1}^{n_k}d_i^2\lambda_i^{1-\theta}\right)^{1/2}\\
&=\left(A_{k,\tau}^{1+\theta}v,v \right)^{1/2}\left(A_{k,\tau}^{1-\theta}w,w\right)^{1/2}=|||v|||_{1+\theta,k,\tau}|||w|||_{1-\theta,k,\tau}.
\end{split}
\end{equation*}
The proof is completed.
\end{proof}
\begin{lemma}\label{lemma4.9}
Let $ u\in H_0^{\alpha/2}(\Omega)$ and $u_h\in V_h$ denote the solution of the variational problems
 $a_{\tau}(u,v)=(g,v)$  $\forall v\in H_0^{\alpha/2}(\Omega)$ and $a_{\tau}(u_h,v_h)=(g,v_h)$ $\forall v_h\in V_h$, respectively.
Then there exists a positive constant $C$ such that
$$||u-u_h||_{L^2(\Omega)}\leq C||u-u_h||_{\tau,\alpha/2}\left( \sup_{g\neq 0}\left\{ \frac{1}{||g||_{L^2(\Omega)}}
\inf_{v_h\in V_h} ||w_g-v_h||_{\tau,\alpha/2} \right\} \right),$$
where $w_g\in H_0^{\alpha/2}(\Omega)$ is the unique solution of   $a_\tau(v,w_g)=(g,v)$ $\forall v\in H_0^{\alpha/2}(\Omega)$.
In particular, if  $w_g\in H^{\alpha}(\Omega)$, we have
\begin{equation*}
\begin{split}
||u-u_h||_{L^2(\Omega)}
&\leq Ch^{\alpha/2}\left(1+\tau^{-1}h^{\alpha}\right)^{1/2}||u-u_h||_{\tau,\alpha/2}.
\end{split}
\end{equation*}
\end{lemma}
\begin{proof}
Since
\begin{equation*}
\begin{split}
||u-u_h||_{L^2(\Omega)}
&=\sup_{g\neq 0}\frac{|(g,u-u_h)|}{||g||_{L^2(\Omega)}}
=\sup_{g\neq 0}\frac{|a_\tau(u-u_h,w_g)|}{||g||_{L^2(\Omega)}}\\
&=\sup_{g\neq 0}\frac{|a_\tau(u-u_h,w_g-v_h)|}{||g||_{L^2(\Omega)}}
\leq \sup_{g\neq 0}\frac{ C||w_g-v_h||_{\tau,\alpha/2}||u-u_h||_{\tau,\alpha/2}}{||g||_{L^2(\Omega)}}\\
&\leq C||u-u_h||_{\tau,\alpha/2}\left( \sup_{g\neq 0}\left\{ \frac{1}{||g||_{L^2(\Omega)}}
\inf_{v_h\in V_h} ||w_g-v_h||_{\tau,\alpha/2} \right\} \right).
\end{split}
\end{equation*}
Using the property of the  interpolation operator $I_h$  \cite{Brenner:08} and (\ref{4.9}), we have
\begin{equation*}
\begin{split}
 ||w_g-v_h||^2_{\tau,\alpha/2}
 &=\tau^{-1}||w_g-v_h||^2_{L^2(\Omega)}\!+\!||w_g-v_h||^2_{{H^{\alpha/2}}(\Omega)}\\
&\leq C_1h^{\alpha}\left(1+\tau^{-1}h^{\alpha}\right)||w_g||^2_{{H^{\alpha}}(\Omega)},
\end{split}
\end{equation*}
i.e.,
\begin{equation*}
\begin{split}
 ||w_g-v_h||_{\tau,\alpha/2}
 &\leq C_2h^{\alpha/2}\left(1+\tau^{-1}h^{\alpha}\right)^{1/2}||w_g||_{{H^{\alpha}}(\Omega)}.
\end{split}
\end{equation*}
According to the above equations and Lemma \ref{lemma3.5}, there exists
\begin{equation*}
\begin{split}
||u-u_h||_{L^2(\Omega)}
&\leq CC_2h^{\alpha/2}\left(1+\tau^{-1}h^{\alpha}\right)^{1/2}||u-u_h||_{\tau,\alpha/2}.
\end{split}
\end{equation*}
The proof is completed.
\end{proof}

\begin{lemma}(Approximation property with $\tau$-norm)\label{lemma4.10}
There exists a positive constant $C$ such that
\begin{equation*}
\begin{split}
||(I-P_{k-1})v||_{\tau,\alpha/2}\leq Ch^{\alpha/2}(1+\tau^{-1}h^{\alpha})^{1/2}|||v|||_{2,k,\tau}~~ \forall v \in V_k.
\end{split}
\end{equation*}
\end{lemma}

\begin{proof}
According to (\ref{4.10}), (\ref{4.6}), (\ref{4.4}), (\ref{4.7}) and  Lemmas \ref{lemma4.8}, \ref{lemma4.9}
\begin{equation*}
\begin{split}
||(I-P_{k-1})v||_{\tau,\alpha/2}^2
&\leq C_1 |||(I-P_{k-1})v|||^2_{1,k,\tau}= C_1(A_{k,\tau}(I-P_{k-1})v,(I-P_{k-1})v)_k \\
&=C_1a_{\tau}((I-P_{k-1})v,(I-P_{k-1})v)=C_1a_{\tau}((I-P_{k-1})v,v)\\
&\leq C_1||(I-P_{k-1})v||_{L^2(\Omega)}|||v|||_{2,k,\tau}\\
&\leq C_1Ch^{\alpha/2}(1+\tau^{-1}h^{\alpha})^{1/2}||(I-P_{k-1})v||_{\tau,\alpha/2}|||v|||_{2,k,\tau}.
\end{split}
\end{equation*}
The proof is completed.
\end{proof}

\begin{definition}\label{definition4.11}
The error operator $E_k:V_k\rightarrow V_k$ is defined recursively by
\end{definition}
\begin{equation*}
\begin{split}
E_1&=0,\\
E_k&=K^m_k[I-(I-E_{k-1})P_{k-1}]K^m_k,
\end{split}
\end{equation*}
where $K_k=I-S_kA_{k,\tau}$, $k\geq 1$ in (\ref{4.8}).

\begin{lemma}\label{lemma4.12}
Let $z,g\in V_k$ satisfy $A_{k,\tau}z=g$ with the initial guess $z_0$.  Then
\begin{equation*}
\begin{split}
E_k(z-z_0)=z-{\rm MG} (k,z_0,g)~~\forall k\geq 1.
\end{split}
\end{equation*}
\end{lemma}
\begin{proof}
The similar arguments can be performed as  \cite{Bank:85,Bramble:87,Brenner:08}, we omit it here.
\end{proof}

\begin{lemma}(Smoothing Property for the V-cycle)\label{lemma4.00010}
\begin{equation*}
\begin{split}
a_{\tau}((I-K_k)K_k^{2m}v,v)\leq \frac {1}{2m}a_{\tau}((I-K_k^{2m})v,v).
\end{split}
\end{equation*}
\end{lemma}
\begin{proof}
The similar arguments can be performed as  \cite{Bank:85,Bramble:87,Brenner:08}, we omit it here.
\end{proof}

\begin{lemma}\label{lemma4.13}
Let $m$ be the number of smoothing steps. Then
\begin{equation}\label{4.011}
\begin{split}
a_{\tau}(E_kv,v)\leq  \frac{C^*}{m+C^*}a_{\tau}(v,v) ~~\forall v\in V_k，
\end{split}
\end{equation}
where $C^*$ is a positive constant independent of $h$ and $\tau$.
\end{lemma}
\begin{proof}
Let  $\gamma=\frac{C^*}{m+C^*}$. We prove (\ref{4.011}) by the mathematical induction. It obviously  holds for $k=1$ by Definition \ref{definition4.11}. Assume  that
\begin{equation*}
  a_{\tau}(E_{k-1}v,v)\leq  \gamma a_{\tau}(v,v).
\end{equation*}
Next we prove that (\ref{4.011}) holds. From   Definition \ref{definition4.11} and the above equation, it yields
\begin{equation*}
\begin{split}
a_{\tau}(E_kv,v)
&\leq C_2(1-\gamma)||(I-P_{k-1})K^m_kv||_{\tau,\alpha/2}^2+\gamma a_{\tau}(K^m_kv,K^m_kv).
\end{split}
\end{equation*}
According to  Lemmas \ref{lemma4.10}, \ref{lemma4.6} and  \ref{lemma4.00010}, we get
\begin{equation*}
\begin{split}
||(I-P_{k-1})K^m_kv||_{\tau,\alpha/2}^2
&\leq C(1+\tau^{-1}h^{\alpha})^2\frac {1}{2m}(a_{\tau}(v,v)-a_{\tau}(K_k^mv,K_k^mv)).
\end{split}
\end{equation*}
Taking  $C^*=\frac{CC_2(1+\tau^{-1}h^{\alpha})^2}{2}$ and using the above equations,  the desired results is obtained.
\end{proof}

\begin{theorem}\label{theorem4.14}[Uniform convergence of V-cycle MGM with fractional $\tau$-norm]
Let $m$ be the number of smoothing steps. Then
$$
||z-{\rm MG} (k,z_0,g)||_{\tau,E}\leq  \frac{C^*}{m+C^*}||z-z_0||_{\tau,E} ~~\forall z\in V_k,
$$
where the time-dependent energy  norm associated with (\ref{4.2}) is defined by
$$||z||_{\tau,E}=\sqrt{a_\tau(z,z)}.$$
\end{theorem}
\begin{proof}
Let $\mu_i$ be the eigenvalues of the operator $E_k$ and $\varphi_i$ be the corresponding eigenfunction satisfying the orthogonal relation $a_\tau(\varphi_i,\varphi_j)_k=\delta_{i,j}.$
Using  Lemma \ref{lemma4.13}, we obtain  $0<\mu_1\leq\mu_1\cdots\mu_{n_k}\leq\gamma$,
where $\gamma=\frac{C^*}{m+C^*}$ is given in (\ref{4.011}).
Let  $v=\sum_{i=1}^{n_k}c_i\varphi_i$, we have
\begin{equation*}
\begin{split}
||E_kv||_{\tau,E}^2
=a_{\tau}(E_kv,E_kv)=\sum_{i=1}^{n_k}c_i^2\mu_i^2\leq\gamma^2 a_{\tau}(v,v).
\end{split}
\end{equation*}
From  Lemma \ref{lemma4.12} and the above equation, the desired results is obtained.
\end{proof}

\begin{remark}
The time-dependent fractional  MGM can be treated as the elliptic equations arising at a  fixed time step  \cite{Bu:15}.
Here, it is   independent of the time step, i.e., $\tau\rightarrow 0$.

\end{remark}
\section{Numerical Results}

We employ the V-cycle MGM  described in Algorithm  \ref{MGM} to solve the resulting system. As for operation counts, the cost of using FFT to would lead to $O(N\log(N)$ \cite{Chan:07},
where $N$  denotes the number of the  grid points. The stopping criterion is taken as
$\frac{||r^{(i)||}}{||r^{(0)}||}<10^{-10},$
where $r^{(i)}$ is the residual vector after $i$ iterations;
and the  number of  iterations $(m_1,m_2)=(1,2)$  and $(\eta_{pre},\eta_{post})=(1/2,1/2)$.
In all tables, the numerical errors are measured by the $ L_2$ norm,  `Rate' denotes the convergence orders.
`CPU' denotes the total CPU time in seconds (s) for solving the resulting discretized  systems;
and `Iter' denotes the average number of iterations required to solve a general linear system $A_{h,\tau}u_h=g_h$ at each time level.

All numerical experiments are programmed in Matlab, and the computations are carried out  on a PC with the configuration:
Inter(R) Core (TM) i5-3470  CPU 3.20 GHZ and 8 GB RAM and a Windows 7 operating system.
\begin{example}\end{example}
Let us consider the  time-dependent tempered fractional  problem (\ref{1.6})
with $\sigma= 3\lambda^\alpha\kappa_{\alpha}$ and  $a< x < b $, $0 < t \leq T$.
Take the exact solution of the equation as $u(x,t)=e^{-t}x^2(1-x/b)^2$, then the corresponding initial and boundary conditions are, respectively,
$u(x,0)=x^2(1-x/b)^2$ and $u(0,t)=u(b,t)=0$; and the forcing function
\begin{equation*}
\begin{split}
 f(x,t)=&-e^{-t}x^2(1-x/b)^2(1-5\lambda^\alpha\kappa_{\alpha})\\
 &-e^{-t}\kappa_{\alpha}\left(e^{-\lambda x}{ _{a}}D_x^{\alpha}\left[e^{\lambda x}x^2(1-x/b)^2\right]
 +e^{\lambda x}{_{x}}D_{b}^{\alpha}\left[e^{-\lambda} x^2(1-x/b)^2\right]\right).
  \end{split}
\end{equation*}
Here the left and right fractional derivatives of the given functions are calculated by the following  Algorithm \ref{AFD}.
\renewcommand{\algorithmicrequire}{ \textbf{Input:}} 
\renewcommand{\algorithmicensure}{ \textbf{Output:}} 

\begin{algorithm}
\caption{Calculating the Left and Right Fractional Derivatives}
\label{AFD}
\begin{algorithmic}[1]

\REQUIRE ~~\\ 
Original function $u(x) \in C^2(a,b)\cap C_0^1(a,b)$ and $\alpha \in (1,2)$\\
\ENSURE ~~\\
Denote the values of numerically calculating $_aD_x^\alpha u(x)$ and $_xD_b^\alpha u(x)$ by $v_l$ and $v_r$\\
The algorithm JacobiGL of generating the nodes and weights of Gauss-Labatto integral with the weighting function $(1-x)^{1-\alpha}$ or $(1+x)^{1-\alpha}$  can be seen in \cite{Hesthaven:07}\\
\STATE $z,w:=$JacobiGL$(1-\alpha,0,100)$
\STATE $v_l:=\frac{1}{\Gamma(2-\alpha)}\left(\frac{x-a}{2}\right)^{2-\alpha}
\sum\limits_{i=1}^{100} \frac{\partial^2 u}{\partial x^2}\left(\frac{x-a}{2}z_i+\frac{x+a}{2}\right)w_i$
\STATE $z,w:=$JacobiGL$(0,1-\alpha,100)$
\STATE $v_r:=\frac{1}{\Gamma(2-\alpha)}\left(\frac{b-x}{2}\right)^{2-\alpha}
\sum\limits_{i=1}^{100} \frac{\partial^2 u}{\partial x^2}\left(\frac{b-x}{2}z_i+\frac{b+x}{2}\right)w_i$
\end{algorithmic}
\end{algorithm}

\begin{table}[h]\fontsize{9.5pt}{12pt}\selectfont
\begin{center}
\caption {MGM to solve the resulting system  (\ref{4.5}) with $\lambda=0.5$, $a=0$, $b=32$, $T=1$ and
$\tau=T/N$, $h=b/M$, $N=M$. } \vspace{5pt}
\begin{tabular*}{\linewidth}{@{\extracolsep{\fill}}*{9}{c}}                                        \hline  
    $N$ &  $\alpha=1.1$ &    Rate &  Iter &      CPU(s)&  $\alpha=1.8$ &     Rate &  Iter &  CPU(s) \\\hline \\
   $2^7$&     4.7035e-03&       ~~&     21&        1.63&     1.5598e-02&     ~~   &    17 &   1.41   \\ 
   $2^8$&     1.1469e-03&   2.0360&     19&        3.70&     3.8736e-03&    2.0096&    14 &   3.06    \\ 
   $2^9$&     2.8262e-04&   2.0208&     16&        8.07&     9.6415e-04&    2.0063&    11 &   6.14     \\ 
$2^{10}$&     7.0031e-05&   2.0128&     14&       19.17&     2.4353e-04&    1.9852&    13 &  18.20      \\\hline 
    \end{tabular*}\label{tab:1}
  \end{center}
\end{table}
From Table \ref{tab:1}, we numerically confirm that the numerical scheme has second-order accuracy and the computational cost is almost $\mathcal{O}(N \mbox{log} N)$ operations.

\begin{example}\label{example5.2}\end{example}
Consider the time-dependent  tempered fractional  problem (\ref{1.6}) on a domain  $0< x < 1 $,  $0<t \leq 1$.
We take the initial condition $u(x,0)=x(1-x)$ with the homogeneous boundary conditions, and  the forcing function is $f(x,t)=0$.
Since the analytic solutions is unknown for Example \ref{example5.2}, the order of the convergence of the numerical results are computed by the following formula
\begin{equation*}
  {\rm Convergence ~Rate}=\frac{\ln \left(||u_{2h}^N-u_{h}^N||_{L_2}/||u_{h}^N-u_{h/2}^N||_{L_2}\right)}{\ln 2}.
\end{equation*}
\begin{table}[h]\fontsize{9.5pt}{12pt}\selectfont
 \begin{center}
  \caption {The $L_2$ errors and convergence orders for  (\ref{4.5}) with   $\sigma= 0$, $\lambda=0.5$ and
$\tau=1/N$, $h=1/M$, $N=M$.} \vspace{5pt}
\begin{tabular*}{\linewidth}{@{\extracolsep{\fill}}*{10}{c}}                                    \hline  
     $N$&    $\alpha=1.1$&          Rate &  $\alpha=1.5$  & Rate         & $\alpha=1.9$  &   Rate    \\\hline   \\
   $2^7$&    5.4283e-05  &               & 8.1597e-06     &              & 3.6689e-06    &       \\
   $2^8$&    2.7895e-05  &  0.9605       & 3.9436e-06     & 1.0490       & 1.2725e-06    & 1.5277 \\
   $2^9$&    1.4239e-05  &  0.9702       & 1.9427e-06     & 1.0214       & 4.2293e-07    & 1.5891  \\
$2^{10}$&    7.2344e-06  &  0.9769       & 9.6532e-07     & 1.0090       &1.3735e-07     & 1.6226   \\ \hline
    \end{tabular*}\label{table:2}
  \end{center}
\end{table}

Table \ref{table:2} shows that the scheme (\ref{4.5}) preserves the desired first order convergence with nonhomogeneous  initial conditions. And it is not possible to reach second-order convergence
because of the weak regularity of the solution in the region close to the initial point and the boundaries.

\section{Conclusions}

In this work we have developed  variational formulations  for the  time-dependent Riesz  tempered fractional  problem.
The stability of the variational formulation, and the Sobolev regularity  of the variational solutions were established.
There are already some uniform  convergence estimates  for using the V-cycle MGM to solve the  time-dependent PDEs, however,
we notice that the proofs are mainly based on  the fixed time step $\tau$. We introduce   and  define   the  fractional $\tau$-norm,   then  the  convergence rate of the V-cycle MGM   is strictly proved with $\tau \rightarrow 0$. The numerical experiments are performed to verify the convergence with only  $\mathcal{O}(N \mbox{log} N)$ complexity by  FFT.
We remark that the corresponding theoretical  can be extended to the time-fractional Feynmann-Kac equation \cite{CDS:17} including the classical parabolic PDEs.
It should be noted that the real challenge is the verification of the {\em weak} solution regularity  for the fractional  operator and this will the subject of future researches.

\bibliographystyle{amsplain}

\begin{thebibliography}{10}

\bibitem{Acosta:17} {\sc  G. Acosta and  J. P.  Borthagaray},   {\em A fractional Laplace equation: Regularity of solutions and finite element approximations},
SIAM J. Numer. Anal.,  \textbf{55}  (2017), pp. 472--495.

\bibitem{Adams:75} {\sc  R. A. Adams},   {\em Sobolev Spaces},
Academic Press, New York, 1975.






\bibitem{Arico:07} {\sc  A. Aric\`{o} and  M. Donatelli},   {\em A V-cycle multigrid for multilevel matrix algebras: proof of optimality},
Numer. Math.,  \textbf{105}  (2007), pp. 511--547.

\bibitem{Arico:04} {\sc  A. Aric\`{o}, M. Donatelli, and S. Serra-Capizzano},
{\em V-cycle optimal convergence for certain (multilevel) structured linear systems},
SIAM J. Matrix Anal. Appl.,  \textbf{26}  (2004), pp. 186--214.


\bibitem{Baeumera:10} {\sc  B. Baeumera and  M. M. Meerschaert},   {\em Tempered stable L\'{e}vy motion and transient super-diffusion},
 J. Comput. Appl. Math.,  \textbf{233}  (2010), pp. 2438-2448.

\bibitem{Bank:81}{\sc R. E.  Bank  and  T.  Dupont},
{\em An optimal order process for solving finite element equations}, Math. Comp., \textbf{153} (1981), pp. 35--51.


\bibitem{Bank:85} {\sc  R. E.  Bank and  C. C.  Douglas},
{\em Sharp estimates for multigrid rates of convergence with general smoothing and acceleration},  SIAM J. Numer. Anal. {\bf22}   (1985), pp. 617--633.



\bibitem{Bolten:15} {\sc M.  Bolten,    M.  Donatelli,   T.  Huckle,  and C.  Kravvaritis},
{\em Generalized grid transfer operators for multigrid methods applied on Toeplitz matrices},  BIT., {\bf55}     (2015), pp.  341--366.


\bibitem{Bramble:87}{\sc J. H.  Bramble  and  J. E.  Pasciak},
{\em New convergence estimates for multigrid algorithms}, Math. Comp., \textbf{49} (1987), pp. 311--329.

\bibitem{Brenner:08} {\sc  S. C. Brenner   and L. R.  Scott},  {\em The Mathematical Theory of Finite Element Methods},  Springer, 2008.

\bibitem{Briggs:00} {\sc  W. L. Briggs,   V. E. Henson, and S. F.  Mccormick},  {\em A Multigrid Tutorial},  SIAM, 2000.


\bibitem{Bu:15} {\sc  W. P. Bu, X. T. Liu,  Y. F. Tang,  and J. Y.  Yang},
{\em Finite element multigrid method for multi-term time fractional advection diffusion equations}, Int. J. Model. Simul. Sci. Comput., \textbf{6} (2015), 1540001.

\bibitem{Bu:14} {\sc  W. P. Bu,  Y. F. Tang,  and J. Y.  Yang},
{\em Galerkin finite element method for two-dimensional Riesz space fractional diffusion  equations  }, J. Comput. Phys., \textbf{276} (2014), pp.  26--38.


\bibitem{Bu:17} {\sc  W. P. Bu,  A. G. Xiao,  and   W. Zeng},  {\em Finite difference/finite element methods for distributed order time fractional diffusion equations},
J. Sci. Comput., \textbf{72}  (2017), pp. 422--441.

\bibitem{Cartea:07}  {\sc \'{A}. Cartea and   D. del-Castillo-Negrete},
{\em  Fluid limit of the continuous-time random walk with general L\'{e}vy
jump distribution functions}, Phys. Rev. E,  76 (2007),  041105.

\bibitem{Celik:17}  {\sc C. \c{C}elik and   M. Duman},
{\em  Finite element method for a symmetric tempered fractional diffusion equation}, Appl. Numer. Math.,  \textbf{120} (2017),  pp. 270--286.


\bibitem{Chan:98}{\sc R. H. Chan, Q. S. Chang, and H. W. Sun}, {\em Multigrid method for ill-conditioned symmetric Toeplitz systems},
 SIAM J. Sci. Comput.,  \textbf{19} (1998), pp. 516--529.

\bibitem{Chan:07}{\sc R. H. Chan and X. Q. Jin}, {\em An Introduction to Iterative Toeplitz Solvers}, SIAM, 2007.

\bibitem{Chen:13}{\sc  M. H. Chen and    W. H. Deng},
{\em  Discretized fractional substantial calculus}, ESAIM: Math. Mod. Numer. Anal.,
\textbf{49} (2015), pp.  373-394.


\bibitem{Chen:15}  {\sc M. H. Chen and  W. H. Deng},
 {\em  High order algorithms for the fractional substantial diffusion equation with truncated Levy flights},
  SIAM J. Sci. Comput.,  \textbf{37} (2015), pp.  A890-A917.


\bibitem{Chen:14}{\sc M. H. Chen, Y. T. Wang, X. Cheng, and  W. H. Deng},
{\em Second-order LOD multigrid method for multidimensional Riesz fractional diffusion equation},
 BIT, \textbf{54} (2014), pp. 623--647.


\bibitem{CD:17} {\sc  M. H. Chen and W. H. Deng},
{\em Convergence analysis of a multigrid method for a nonlocal model},
SIAM J. Matrix Anal. Appl.,  \textbf{38}  (2017), pp. 869--890.

\bibitem{CDS:17}{\sc  M. H. Chen,    W. H. Deng, and S. Serra-Capizzano},
{\em  Uniform convergence of V-cycle multigrid algorithms for two-Dimensional fractional Feynman-Kac equation},
J. Sci. Comput., \textbf{74}  (2018), pp. 1034--1059.



\bibitem{CN:16} {\sc  L. Chen, R. H. Nochetto, E. Ot\'{a}rola,  and A. J. Salgado},
{\em Multilevel methods for nonuniformly elliptic operators and fractional diffusion},
Math. Comp., \textbf{85} (2016), pp. 2583--2607.



\bibitem{Deng:16} {\sc W. H. Deng and  Z. J. Zhang},
{\em Variational formulation and efficient implementation for solving the tempered  fractional problems},
Numer. Methods Partial Differential Eq., 34 (2018), pp. 1224--1257.



\bibitem{del-Castillo-Negrete:09} {\sc D. del-Castillo-Negrete},  {\em Truncation effects in superdiffusive front propagation with L\'{e}vy flights},
Phys. Rev. E, 79 (2009),  031120.


\bibitem{Deng:08}  {\sc W. H. Deng},  {\em Finite element method for the space and time
fractional Fokker-Planck equation},  SIAM J. Numer. Anal.,  47  (2008),  pp. 204--226.

\bibitem{Douglas:79} {\sc J. Douglas, T. Dupont, and  R. E. Ewing},
{\em Incomplete iteration for time-stepping a Galerkin method for a quasilinear parabolic problem},
SIAM J. Numer. Anal.,  \textbf{16}  (1979),  pp. 503--522.



\bibitem{Ervin:07}  {\sc V. J. Ervin, N. Heuer, and J. P. Roop},  {\em Numerical approximation of a time dependent, nonlinear, space-fractional diffusion equation},
SIAM J. Numer. Anal.,  45  (2007),  pp. 572--591.

\bibitem{Ervin:05} {\sc V. J. Ervin and J. P. Roop}, {\em Variational formulation for the stationary
fractional advection dispersion equation},  Numer. Meth. Part. D. E., \textbf{22} (2006), pp. 558-576.



\bibitem{Ervin:18} {\sc V. J. Ervin, N. Heuer and J. P. Roop}, {\em Regularity of the solution to 1-D fractional order diffusion equation},  
Math. Compt.,  \textbf{87}  (2018), pp. 2273--2294.




\bibitem{Hackbusch:85}  {\sc W. Hackbusch},   {\em Multigird Methods and Applications},  Springer-Verlag, Berlin, 1985.

\bibitem{Hesthaven:07}{\sc J. S. Hesthaven and T. Warburton}, {\em Nodal Discontinuous Galerkin Methods: Algorithms, Analysis, and Applications},
 Springer, 2007.

\bibitem{Horn:13} {\sc     R. A. Horn and  C. R. Johnson},   {\em Matrix Analysis},
 New York: Cambridge University Press, 2013.


\bibitem{Fiorentino:96}  {\sc G. Fiorentino and   S.  Serra},
  {\em Multigrid methods for symmetric positive definite block Toeplitz matrices with nonnegative generating functions},
SIAM J.  Sci. Comput.,  \textbf{17}  (1996), pp. 1068--1081.


\bibitem{Friedman:64} {\sc A. Friedman},   {\em Partial Differential Equation of Parabolic Type},
 Prentice: Prentice-Hall, Inc., 1964.


\bibitem{Gilbarg:97} {\sc D.  Gilbarg and N. S. Trudinger}, {\em Elliptic Partial Differential Equations of Second Order},  Springer, 1997.


\bibitem{Jia:14} {\sc  R. Q.  Jia},   {\em Applications of multigrid algorithms to finite difference schemes for elliptic equations with variable coefficients},
  SIAM J. Sci. Comput.,  \textbf{36} (2014), pp.  A1140-A1162.

\bibitem{Jiang:15} {\sc  Y. J. Jiang and  X. J.  Xu},   {\em Multigrid methods for space fractional partial differential equations},
 J. Comput. Phys., \textbf{302} (2015), pp.  374--392.


\bibitem{Jin:15} {\sc  B. T. Jin, R. Lazarov, J. Pasciak and  W.  Rundell},   {\em Variational formulation of problem involving fractional order differential operators},
Math. Compt.,  \textbf{84}  (2015), pp. 2665--2700.


\bibitem{Jin:16} {\sc  B. T. Jin, R. Lazarov, and  Z.  Zhou},   {\em A Petrov-Galerkin finite element method for fractional convection-diffusion equations},
SIAM J. Numer. Anal.,  \textbf{54}  (2016), pp. 481--503.



\bibitem{Koponen:95} {\sc I. Koponen}, {\em Analytic approach to the problem of convergence of
truncated L\'{e}vy flights towards the Gaussian stochastic process},
Phys. Rev. E, 52 (1995),  pp. 1197--1195.



\bibitem{Li:13} {\sc  L. M.  Li, D. Xu, and M.  Luo},
{\em Alternating direction implicit Galerkin finite element method for the two-dimensional fractional diffusion-wave equation},
J. Comput. Phys., \textbf{255} (2013), pp.  471--485.

\bibitem{Mantegna:94} {\sc  R. N. Mantegna and H. E.  Stanley},
{\em Stochastic process with ultraslow convergence to a Gaussian: the truncated L\'{e}vy flight}, Phys. Rev. Lett., \textbf{73} (1994), pp.  2946--2949.


\bibitem{Metzler:00} {\sc  R. Metzler and J. Klafter},
{\em The random walk’s guide to anomalous diffusion: A fractional dynamics approach}, Phys. Rep., \textbf{339} (2000), pp.  1--77.



\bibitem{Miller:93}  {\sc K. S. Miller and  B. Ross}, {\em  An Introduction to the Fractional Calculus and Fractional
Differential Equations}, Wiley-Interscience Publication, USA, 1993.


\bibitem{Pang:12} {\sc  H. Pang and H.  Sun},
{\em Multigrid method for fractional diffusion equations}, J. Comput. Phys., \textbf{231} (2012), pp.  693--703.


\bibitem{Podlubny:99}{\sc  I. Podlubny}, {\em Fractional Differential Equations}, New York: Academic Press, 1999.


\bibitem{Quarteroni:94}{\sc  A. Quarteroni and A. Valli}, {\em Numerical Approximation of Partial Differential Equations}, Springer, 1994.



\bibitem{Ren:17} {\sc J. C. Ren,  X. N. Long, S. P. Mao, and   J.  W. Zhang},  {\em Superconvergence of finite element approximations for the fractional diffusion-wave equation},
J. Sci. Comput., \textbf{72}  (2017), pp. 917--935.


\bibitem{Ros-Oton:17} {\sc X. Ros-Oton and   J. Serra},  {\em The Dirichlet problem for the fractional Laplacian: Regularity up to the boundary},
J. Math. Pures Appl., \textbf{101}  (2014), pp. 275--302.

\bibitem{Rudin:87} {\sc W.  Rudin}, {\em Real and Complex Analysis},  McGraw-Hill, Inc, 1987.



\bibitem{Thomee:97} {\sc V.  Thom\'{e}e}, {\em Galerkin Finite Element Methods for Parabolic Problems},  Springer, 1997.









\bibitem{Wang:13} {\sc  H. Wang and  D. P.  Yang},   {\em Wellposedness of variable-coefficient conservative fractional elliptic differential equations},
SIAM J. Numer. Anal.,  \textbf{51}  (2013), pp. 1088--1107.




\bibitem{Xu:02} {\sc   J.  Xu and L. Zikatanov},
{\em The method of alternating projections and the method of subspace corrections in Hilbert space},
J. Am. Math. Soc.,  \textbf{15} (2002), pp. 573--597.

\bibitem{Yue:17}  {\sc X. Yue, W. Bu, S. Shu, M. Liu and S. Wang},
 {\em  Fully finite element adaptive algebraic multigrid method for time-space Caputo-Riesz fractional diffusion equations},
arXiv:1707.08345.

Fully Finite Element Adaptive Algebraic Multigrid Method for Time-Space Caputo-Riesz Fractional Diffusion Equations

\bibitem{Zhou:13} {\sc  Z. Q. Zhou and  H. G. Wu},   {\em Finite element multigrid method for the boundary value problem of fractional advection dispersion equation},
J. Appl. Math.,    (2013), 385463.



\bibitem{Zeng:13}  {\sc F. H. Zeng, C. P. Li, F. W. Liu, and I. Turner},
 {\em  The use of finite difference/element approaches for solving the time-fractional subdiffusion equation},
  SIAM J. Sci. Comput.,  \textbf{35} (2013), pp.  A2976-A3000.



\end{thebibliography}

\end{document}